\documentclass[]{interact}

\usepackage{epstopdf}
\usepackage[caption=false]{subfig}

\usepackage[numbers,sort&compress]{natbib}
\bibpunct[, ]{[}{]}{,}{n}{,}{,}

\makeatletter
\def\NAT@def@citea{\def\@citea{\NAT@separator}}
\makeatother

\theoremstyle{plain}
\newtheorem{theorem}{Theorem}[section]
\newtheorem{lemma}[theorem]{Lemma}

\theoremstyle{definition}
\newtheorem{definition}[theorem]{Definition}

\theoremstyle{remark}
\newtheorem{remark}{Remark}

\usepackage{mathtools}
\usepackage{xcolor}
\newcommand{\R}{{\mathbb{R}}}
\newcommand{\N}{{\mathbb{N}}}
\newcommand{\U}{\mathcal{U}}

\newcommand{\dimx}{{n_x}}
 
\newcommand{\dimU}{{n_u}}
\newcommand{\dimpi}{{n_{\pi}}} 
\newcommand{\numberconstraints}{{m}}

\newcommand{\bmx}{{\bm{x}}}
\newcommand{\bmu}{{\bm{u}}}
\newcommand{\bmpi}{{\bm{\pi}}}
\newcommand{\bmb}{{\bm{b}}}
\newcommand{\bmv}{{\bm{v}}}

\newcommand{\setforx}{X} 

\newcommand{\setforpi}{{X^{\pi}}} 
\newcommand{\objectivefunctioninx}{f}
\newcommand{\objectivefunctioninpi}{{f^\pi}} 

\newcommand{\regretfunctioninpi}{r}  
\newcommand{\optimalregretvalueinpi}{r}
\newcommand{\optimalvaluefunction}{{\phi}}

\newcommand{\Aforpi}{{A^{\pi}}}

\newcommand{\decisionrulevectorparameter}{{\bm{\pi}^0}} 
\newcommand{\decisionrulematrixparameter}{\Pi}

\newcommand{\extremalpointsofU}{\U^{extr}}
\usepackage{algorithm}
\usepackage{algpseudocode}

\usepackage{amsthm}
\usepackage{thmtools}
\usepackage{amsmath}
\usepackage{amssymb}
\usepackage{nicefrac}

\usepackage[super]{nth}
\usepackage{changes}

\begin{document}
	
	\articletype{ARTICLE TEMPLATE}
	
	\title{An Adaptive Three-Stage Algorithm For Solving Adjustable Min-Max-Regret Problems}
	
	\author{
		\name{Kerstin Schneider\textsuperscript{a}\thanks{CONTACT Kerstin Schneider. Email: kerstin.schneider@itwm.fraunhofer.de},  Helene Krieg\textsuperscript{a}, Dimitri Nowak\textsuperscript{a} and Karl-Heinz K\"ufer\textsuperscript{a}}
		\affil{\textsuperscript{a}Fraunhofer Institute for Industrial Mathematics ITWM, Kaiserslautern, Germany}
	}
	
	\maketitle
	
	\begin{abstract}
        This work uniquely combines an affine linear decision rule known from adjustable robustness with min-max-regret robustness. By doing so, the advantages of both concepts can be obtained with an adjustable solution that is not over-conservative. 
        This combination results in a bilevel optimization problem. For solving this problem, a three-stage algorithm which uses adaptive discretization of the uncertainty set via two criteria is presented  
        and its convergence is proven.
        The algorithm is applicable for an example of 
        optimizing a robust pump operation plan for a drinking water supply system facing uncertain demand. The algorithm shows a notable ability to scale, presenting an opportunity to solve larger instances that might challenge existing optimization approaches. 
	\end{abstract}
	

	\section{Introduction}

        Robust optimization has emerged as a powerful framework for addressing optimization problems under uncertainty, a pervasive challenge in a multitude of practical applications. This approach is particularly advantageous when the uncertainty distribution is unknown, yet the set of potential scenarios is identifiable. Robust optimization strategies enable the formulation of solutions resilient to a range of uncertain conditions without the need for precise probability distributions.

        The robust optimization literature is rich and diverse, with foundational and advanced contributions from numerous researchers. Seminal works by Ben-Tal et al. \cite{BenTal.1998, BenTal.1999, BenTal.2000, BenTal.2007, BenTal.2009, BenTal.2015} and Bertsimas et al. \cite{Bertsimas.2011, Bertsimas.2016} have laid the groundwork for understanding and applying robust optimization across various contexts. Gabrel et al. \cite{Gabrel.2014} and Sözüer and Thiele's review \cite{Sözüer2016} provide comprehensive overviews of the state of robust optimization.

        The application of robust optimization principles has been demonstrated in various domains, including production and inventory management \cite{BenTal.2004, Qiu.2022}, timetabling \cite{Liebchen., Cicerone.2009, Schobel.2009}, flood protection \cite{Postek.2019}, engineering \cite{Bertsimas.2011, Isenberg.2021}, portfolio optimization \cite{Bertsimas.2011, Xidonas.2020}, and the design, operation, and management of water distribution systems \cite{Jung.2019}. These applications highlight the method's ability to provide robust solutions in unpredictable environments, showcasing its broad applicability and effectiveness in navigating the uncertainties inherent in complex decision-making processes.

        Within the robust optimization field, min-max-regret and adjustable robustness are two prominent methodologies. Min-max-regret focuses on minimizing the worst-case regret, which is the difference between the outcome of a decision and the best possible outcome in hindsight, across all uncertain scenarios \cite{Kouvelis.1997, Assavapokee.2008}. This concept is also known as deviation robust optimization \cite{Kouvelis.1997}. The regret of a scenario is the difference in the objective values of the robust solution to the deterministic solution of this scenario. 

        Adjustable robustness offers a dynamic approach where decisions can be staged, with some decisions made after partial uncertainty resolution, often modeled by an affine linear decision rule \cite{BenTal.2004, Delage.2015, Gorissen.2015, Marandi.2018, Postek.2019}. The framework of adjustable robust optimization distinguishes between \emph{here-and-now} and \emph{wait-and-see} decisions. Here-and-now decisions cannot depend on the uncertainty and have to be made before the uncertainty is revealed, whereas wait-and-see decisions can depend on (parts of) the uncertainty and are made when actual data is known \cite{BenTal.2004}. Adjustable robust problems with an affine linear decision rule are referred to as \emph{affinely adjustable robustness} \cite{BenTal.2004}. 

        The adjustable robust solution is usually less conservative than a classical robust solution due to the adjustability, but in both concepts, a worst-case-cost approach is taken \cite{BenTal.2004}. An approach for solving nonlinear robust optimization problems within this framework has been proposed in \cite{Isenberg.2021}, with an implementation available for Pyomo called PyROS \cite{Bynum.2021, Hart.2011}.
        
        The difficulty in min-max-regret robustness is that it results in such a bilevel optimization problem, where the first stage involves minimizing the maximal regret over all scenarios, and the second stage is to obtain the perfect information solution. 

        An overview of bilevel optimization can be seen in works by Dempe \cite{Dempe.2002, Dempe.2020}. Further, uncertainty can also occur in bilevel problems, Beck et al. \cite{Beck.2023} provide a survey on bilevel robust optimization. Specific algorithms for solving bilevel min-max-regret problems are discussed in the literature. For instance, Assavapokee et al. \cite{Assavapokee.2008} present an algorithm designed for problems with a mixed integer linear program in the first stage and a linear program in the second stage. Additionally, an algorithm for solving continuous min-max problems with finitely many constraints is proposed by Shimizu \cite{Shimizu.1980}.

        In this paper, adjustable robustness is combined with the min-max-regret approach into a new optimization problem that incorporates an affine linear decision rule. The considered problems have nonlinear objective functions and linear constraints with right-hand-side uncertainty. By integrating adjustability into the problem, the solution can respond to uncertainties as they become known. Since a min-max-regret approach is considered, the resulting decisions are less conservative than with a worst-case-cost approach. 

        The bilevel optimization problem formulated here has semi-infinite constraints, which is common in robust optimization problems with infinite uncertainty sets. To tackle such constraints, the authors utilize adaptive discretization methods, building upon the foundational algorithm by Blankenship and Falk \cite{Blankenship.1976}. These methods have been beneficial in a variety of applications, as evidenced by the work of Krieg et al. \cite{Krieg.2022} and Seidel and K\"ufer \cite{Seidel.2022}.

        This paper introduces an adaptive three-stage algorithm designed to solve the bilevel optimization problem by successively augmenting the discretization of the uncertainty set by scenarios, for which the current iterate solution violates feasibility or optimality. A numerical example is provided to illustrate the algorithm’s ability to scale. This scalability is critical in extending the applicability of robust optimization to more complex and sizeable problems. The authors prove convergence of the algorithm under certain assumptions on the problem functions.

        The structure of the remainder of this work is organized as follows: Section \ref{section:problem_description} fuses elements of adjustable robustness with the min-max-regret framework to formulate an optimization problem that utilizes an affine linear decision rule within a min-max-regret approach. Section \ref{section:algorithm} details the innovative adaptive three-stage algorithm developed to solve this problem and provides a proof of its convergence. Section \ref{section:results} presents numerical examples that illustrate the algorithm's performance and includes a comparative analysis with the affine adjustable robustness approach that focuses on worst-case optimization. The advantages and possible improvements of the presented algorithm are discussed in Section \ref{section:discussion} and the paper concludes with Section \ref{section:conclusion}, where the main findings are summarized.

        The results presented in this work are also contained in the PhD thesis of Kerstin Schneider \cite{Schneider.2024}, which has been submitted at the end of March 2024.

	\section{Problem description}\label{section:problem_description}
	Let $X\subseteq \R^\dimx$ and $\setforpi \subseteq \R^\dimpi$, $n_x \in \mathbb{N}$, be compact sets and consider an interval-shaped uncertainty set 
	\begin{equation}
		\U \coloneqq \left[\bmu_{min}, \bmu_{max} \right]^\dimU \subseteq \R^\dimU.
	\end{equation}	
    For $n_u \in \mathbb{N}$, denote the set of all extremal scenarios of $\U$ by
	\begin{equation}
		\extremalpointsofU \coloneqq \{ \bmu \in \U | u_i \in \{u_i^{min}, u_i^{max}\}, \ i=1,..., \dimU\}.
	\end{equation} 
	Throughout this work, the following uncertain optimization problem is considered:
	\begin{equation}\label{opt_problem:without_robustness_concept}
		\begin{aligned}
			\min_{\bm{x}\in X} \quad & \objectivefunctioninx(\bm{x}) \\
			\text{s.t. } \ & A \bm{x} \leq \bm{b}(\bm{u}), \quad \forall \bm{u} \in \U,
		\end{aligned}
	\end{equation} 
	with a continuous and convex function $f: \setforx \rightarrow \R$, and $m \in \mathbb{N}$ linear constraints defined by a matrix $A \in \R^{\numberconstraints \times \dimx}$ and a vector $ \bm{b}(\bm{u}) \in \R^{\numberconstraints}$. It is assumed that the entries of $\bm{b}$ can depend linearly on the uncertain parameter $\bm{u} \in \U$.

    \subsection{Introduction of an affine decision rule}
	In the context of affine adjustable robustness, the decision variable $\bm{x}$ is replaced by an affine linear decision rule 
	\begin{equation}
		\bm{x} = \decisionrulevectorparameter + \decisionrulematrixparameter \bm{u}.
		\label{eq:decision_rule}
	\end{equation}
    Doing so, one can optimize for the parameters $\decisionrulevectorparameter$ and $\decisionrulematrixparameter$. When the uncertainty $\bm{u} \in \U$ is revealed, the solution $\bm{x}\in \setforx$ can be obtained via the decision rule \eqref{eq:decision_rule}. This has the advantage, that the decision is adjusted to at least some part of the uncertain parameters as soon as they get known. 
    
	The new decision variables are the vector $\decisionrulevectorparameter \in \R^{\dimx}$ and the matrix $\Pi \in \R^{\dimx \times \dimU}$. Thereby, the assumption is made that the entries $\pi_{ij}$ of $\decisionrulematrixparameter$ are restricted to $|\pi_{ij}| \leq N$ for some given amount $N \in \R$. This restriction can be motivated by the consideration that the adjustment to a realized scenario must not be arbitrarily large. The number $N$ can be chosen large enough, such that this is not a severe restriciton. 
    For easier notation, the new decision variables are summarised in the vector $\bm{\pi}$ that contains all entries from this vector-matrix-pair $(\decisionrulevectorparameter, \decisionrulematrixparameter)$, i.e.,\ $\bm{\pi} \in \setforpi \subseteq \R^{\dimpi}$ with $\dimpi = (\dimU +1) \dimx$. Since the entries of $\decisionrulematrixparameter$ are restricted by $N$ and the decision $\bmx$ resulting from the decision rule \eqref{eq:decision_rule} must be in the compact set $\setforx$, the entries of $\decisionrulevectorparameter$ are also restricted. Hence, the set $\setforpi \subset \R^\dimpi$ is compact.
 
	Next, the objective is reformulated in terms of the  decision rule \eqref{eq:decision_rule}
	\begin{equation}
		\objectivefunctioninpi(\bm{\pi}, \bm{u}) \coloneqq \objectivefunctioninx(\decisionrulevectorparameter + \decisionrulematrixparameter \bm{u}).
	\end{equation}
	Further, define the matrix $\Aforpi(\bm{u}) \in \R^{\numberconstraints \times \dimpi}$ with entries that can depend linearly on $\bm{u} \in \U$ such that it fulfills the following equation:
	\begin{equation}
		\Aforpi(\bm{u}) \cdot \bm{\pi} = A\cdot (\decisionrulevectorparameter + \decisionrulematrixparameter \bm{u}).
	\end{equation}
	The `$\cdot$' in the equation above indicates a matrix-vector-multiplication.
 
	Replacing $\bm{x}$ by the affine linear decision rule \eqref{eq:decision_rule} and using these definitions results in the following  affinely adjustable robust 
    problem derived from Problem \eqref{opt_problem:without_robustness_concept}:
	\begin{equation}
		\label{opt_problem:affine_adjustable}
		\begin{aligned}
			\min_{\bm{\pi} \in \setforpi} \max_{\bm{u}\in \U} \quad & \objectivefunctioninpi(\bm{\pi}, \bm{u}) \\
			\text{s.t. } \ & \Aforpi(\bm{u}) \bmpi  \leq \bm{b}(\bm{u}) \quad \forall \bm{u} \in \U.
		\end{aligned}
	\end{equation}
	    As the maximum over all uncertain scenarios is minimized in Problem \eqref{opt_problem:affine_adjustable}, this is still a worst-case-approach which makes sure that the maximal possible value of the objective function is as small as possible. However, the possibility that other than worst-case scenarios occur is not taken into account by this objective. As a result, a solution to Problem \eqref{opt_problem:affine_adjustable} can result in very high function values when evaluated with respect to such a non-worst-case scenario. 
        As a relief to this issue, \eqref{opt_problem:affine_adjustable} is combined with a min-max-regret approach in the following.

\subsection{The robust regret objective}
	The regret $r(\bm{\pi}, \bm{u})$ of the decision $\bm{\pi}$ in the scenario $\bm{u}$ is defined as the difference of the objective value of the robust solution $\bm{\pi}$ in case of scenario $\bm{u}$ to the objective value of the perfect information solution $\bm{\pi}^*(\bm{u})$ in this scenario. 
	The perfect information solution $\bm{\pi}^*(u)$ solves the deterministic problem for the fixed scenario $\bm{u}$:
	\begin{equation}\label{opt_problem:lower_level_in_pi}
		\begin{aligned}
			\min_{\bm{\pi}^*(\bm{u}) \in \setforpi} \quad & \objectivefunctioninpi(\bm{\pi}^*(\bm{u}), \bm{u}) \\
			\text{s.t.} \quad & \Aforpi(\bm{u}) \pi^*(\bm{u}) \leq b(\bm{u}).
		\end{aligned}
	\end{equation}
	Since only one scenario is considered in Problem \eqref{opt_problem:lower_level_in_pi}, $\bm{\pi}^*(\bm{u})$ can be replaced by $\bm{x}^*(\bm{u})$ using the decision rule \eqref{eq:decision_rule}. Then, the affine adjustable min-max-regret problem is

	\begin{equation}
		\begin{aligned}
			\min_{\bm{\pi} \in \setforpi} \ \max_{\bm{u}' \in \mathcal{U}} \ \ 
   & \objectivefunctioninpi(\bm{\pi}, \bm{u}') - \objectivefunctioninpi(\bm{\pi}^*(\bm{u}'), \bm{u}') \\
			\text{s.t. } & \Aforpi (\bm{u}) \bm{\pi} \leq \bm{b}(\bm{u}) \quad \forall \bm{u} \in \mathcal{U} \\
			& \min_{\bm{x}^*(\bm{u}') \in \setforx} \quad \objectivefunctioninx(\bm{x}^*(\bm{u}')) \\
			& \quad \ \text{  s.t.  } \quad \  A \bm{x}^*(\bm{u}') \leq \bm{b}(\bm{u}')
		\end{aligned} \label{opt_problem:minmaxregret_lower_level_without_decision_rule} \tag{$P_{\text{Min-Max-Regret}}$}
	\end{equation}
	
	This problem is a two-stage optimization problem. Beyond that, the upper level problem has a semi-infinite constraint. 
	In the following, a three-stage adaptive algorithm for solving these kind of problems is introduced for the case of a continuous and convex objective function $\objectivefunctioninx$.
 
 Similar to \cite{Bank.1982}, the \emph{feasible set mapping} of problem \ref{opt_problem:subproblem_min_cost}  
	is denoted by 
	\begin{equation} \label{eq:feasible_set_lower_level_problem}
		F: \U \rightrightarrows X\subset \R^{\dimx}, \quad \bmu \mapsto \left\{ \left. \bmx \in \setforx \right| \left( A \bmx - \bmb(\bmu)\right)_i \leq 0, \ 1 \leq i \leq \numberconstraints \right\}.
	\end{equation}
 Then, the regret of a decision $\bmpi\in \setforpi$ and a scenario $\bmu \in \U$ is defined by
 \begin{equation}
     \regretfunctioninpi(\bmpi, \bmu) \coloneqq \objectivefunctioninpi(\bmpi, \bmu) - \min_{\bmx^*(\bmu) \in F(\bmu)} \objectivefunctioninx(\bmx^*(\bmu)).
 \end{equation}
	
	\section{Algorithm}\label{section:algorithm}
	The algorithm introduced in this chapter is an adaptive discretization algorithm consisting of three stages that finds approximate solutions of bilevel problems with a semi-infinite constraint like \ref{opt_problem:minmaxregret_lower_level_without_decision_rule}.
 
    In the following, discretizations of the uncertainty set $\U$ are denoted with a dot ($\dot{\U}$). Figure \ref{fig:3_stage_algo} provides an overview of the different stages of the algorithm. First, an initial discretization $\dot{\U}^1$ is chosen and the perfect-information-problem \eqref{opt_problem:lower_level_in_pi} which actually is the lower level problem of \ref{opt_problem:minmaxregret_lower_level_without_decision_rule},
	\begin{equation}
		\begin{aligned}
			\min_{\bmx(\bmu)\in \setforx} \quad & \objectivefunctioninx(\bmx(\bmu)) \\
			\text{s.t. } & A \bmx(\bmu) \leq \bmb(\bmu),
		\end{aligned}
		\tag{$P_{\text{Lower Level}}(\bmu)$} \label{opt_problem:subproblem_min_cost}
	\end{equation}
	is solved for every scenario $\bmu\in \dot{\U}^1$ in the initial discretization.
	Then, the \emph{discretized min-max-regret problem}
	\begin{equation}
		\begin{aligned}
			\min_{\bmpi \in \setforpi} \max_{\bmu' \in \dot{\mathcal{U}}} \ \  & \objectivefunctioninpi(\bm{\pi}, \bm{u}') - \objectivefunctioninpi(\bm{\pi}^*(\bm{u}'), \bm{u}') \\
			\text{s.t. } & \Aforpi(\bmu) \bmpi \leq \bmb (\bmu), \quad \forall \bmu \in \dot{\mathcal{U}} \\
			& \min_{\bmx^*(\bmu') \in \setforx} \ \  f(\bmx^*(\bmu')) \\
			& \quad \quad \  \text{s.t. } \  A \bmx^*(\bmu') \leq \bmb(\bmu')
		\end{aligned}
		\tag{$P_{\text{Min-Max-Regret}}(\dot{\mathcal{U}})$}\label{opt_prob:subproblem_min_max_regret}
	\end{equation}
	is solved for the initial discretization $\dot{\U}^1$. The solution $\bmpi^1$ of this problem is feasible for every $\bmu \in \dot{\U}^1$. The maximal regret of $\bmpi^1$ with respect to $\dot{\U}^1$ is denoted by $r^1$, i.e.,
	\begin{equation}
		r^k \coloneqq \max_{\bmu \in \dot{\U}^k} r(\bmpi^k, \bmu).
	\end{equation} 
	The next question is, whether $\bmpi^1$ is feasible for the whole uncertainty set $\U$.
	This is checked in the second stage of the algorithm, where the \emph{max-infeasibility problem}
	\begin{equation}
		\max_{\bmu\in \mathcal{U}}  \ \max_{1\leq j \leq \numberconstraints} \{(\Aforpi(\bmu)\bmpi - \bmb(\bmu))_j\}
		\tag{$P_{\text{Max-Infeasibility}}(\bmpi, \U)$}\label{opt_prob:subproblem_max_infeasibility},
	\end{equation}
	is solved for $\bmpi^1$. The solution of this problem is denoted by $\bmu^k$. With these $\bmpi^1$ and $\bmu^1$, the \emph{feasibility criterion} is tested.
	\begin{definition}(Feasibility criterion.)
		The iterates $\bmpi^k$ and $\bmu^k$ fulfill the \emph{feasibility criterion}, if
		\begin{equation}
			\max_{1\leq j\leq \numberconstraints} \left\{ \left( \Aforpi(\bmu^k)\bmpi^k-\bmb(\bmu^k) \right)_j \right\} \leq 0. \label{eq:feasibility_criterion}
		\end{equation}	 
	\end{definition}
	If the feasibility criterion is violated, there exists $\bmu\in \U$ for which $\bmpi^1$ is not feasible and $\bmu^1$ is the scenario that violates one of the linear constraints the most. In this case, this scenario is added to the discretization of $\U$, i.e.,\ $\dot{\U}^2 = \dot{\U}^1 \cup \{\bmu^1\}$. The algorithm then returns to the \nth{1} stage, where first $P_{\text{Lower Level}}(\bmu^1)$ and afterwards $P_{\text{Min-Max-Regret}}(\dot{\mathcal{U}}^2)$ is solved. This loop is repeated until the current iterates $\bmpi^k$ and $\bmu^k$ fulfill the feasibility criterion.
	When this is the case, the algorithm enters the \nth{3} stage, where the \emph{max-regret problem}
        \begin{equation}
		\begin{aligned}
			\max_{\bmu \in \mathcal{U}, \bmx^*(\bmu) \in \setforx} \quad & \objectivefunctioninpi(\bmpi, \bmu) - \objectivefunctioninx(\bmx^*(\bmu))\\
			\text{s.t. }\quad  & A \bmx^*(\bmu) \leq \bmb(\bmu)
		\end{aligned}
		\tag{$P_{\text{Max-Regret}}(\bmpi, \U)$}\label{opt_prob:subproblem_max_regret}
	\end{equation}
	is solved in order to find the scenario $\bmu^k \in \U$, for which the current solution $\bmpi^k$ of the discretized min-max-regret problem has the largest regret. The maximal regret obtained from this problem and the maximal regret over the discretized uncertainty set $\dot{\U}^k$ are compared in the \emph{$\varepsilon$-termination criterion}:
	\begin{definition}\label{def:termination_criterion}($\varepsilon$-termination criterion.) 
		For some given $\varepsilon>0$, the iterates $\bmpi^k$ and $\bmu^k$ fulfill the $\varepsilon$-\emph{termination criterion}
		, if
		\begin{equation}
			\regretfunctioninpi(\bmpi^k, \bmu^k) - \regretfunctioninpi^k  < \varepsilon \label{eq:termination_criterion}.
		\end{equation}
	\end{definition}
	If the regrets differ by more than $\varepsilon$, the scenario with the largest regret, $\bmu^k$, is added to the discretization of the uncertainty set, i.e.,\ $\dot{\U}^{k+1} = \dot{\U}^k \cup \{\bmu^k\}$ and the algorithm returns to the \nth{1} stage. Otherwise, the current solution $\bmpi^k$ is accepted as approximate solution of the min-max-regret problem \ref{opt_problem:minmaxregret_lower_level_without_decision_rule} and the algorithm terminates. The steps of the algorithm are stated in Algorithm \ref{alg:3_stage_algo}.
	
	\begin{algorithm}[H] 
		\caption{Adaptive 3-stage algorithm for min-max-regret robustness.} \label{alg:3_stage_algo}
		\begin{algorithmic}[1]
			\State Initial discretization: $\dot{\mathcal{U}}^1 \subset \U, k=1$
			\State $\varepsilon$-termination criterion $=$ False
			\State Solve $P_{\text{Lower Level}}(\bmu)$ for every $\bmu \in \dot{\mathcal{U}}^1$  $\rightarrow$ solutions used to solve $P_{\text{Min-Max-Regret}}(\dot{\mathcal{U}}^1)$ in next step \label{ln:initial_solve_lower_level}
			\State Solve  $P_{\text{Min-Max-Regret}}(\dot{\mathcal{U}}^1)$ $\rightarrow$ solution $\bmpi^1$
			\State Solve $P_{\text{Max-Infeasibility}}(\bmpi^1, \U)$ $\rightarrow$ solution $\bmu^1\in \U$
			\While{$\varepsilon-$termination criterion \eqref{eq:termination_criterion} is not fulfilled}
			\State $k=k+1$ \label{ln:start_2nd_stage}
			\State $\dot{\mathcal{U}}^{k} = \dot{\mathcal{U}}^{k-1} \cup \{\bmu^{k-1}\}$
			\State Solve $P_{\text{Lower Level}}(\bmu^{k-1})$ $\rightarrow$ solution used to solve $P_{\text{Min-Max-Regret}}(\dot{\mathcal{U}}^k)$ in next step \label{ln:SolvePLowerLevel}
			\State Solve $P_{\text{Min-Max-Regret}}(\dot{\mathcal{U}}^k)$ $\rightarrow$ solution $\bmpi^k$, optimal value $\regretfunctioninpi^k$ \label{ln:SolvePMinMaxRegretDiscretized}
			\State Solve $P_{\text{Max-Infeasibility}}(\bmpi^k , \U)$ $\rightarrow$ solution $\bmu^{k} \in \U$ \label{ln:SolvePFeas}
			\If {feasibility criterion \eqref{eq:feasibility_criterion} is not fulfilled} go to \ref{ln:start_2nd_stage} \label{ln:startif}
			\EndIf \label{ln:endif}
			\State Solve $P_{\text{Max-Regret}}(\bmpi^{k},\U)$ $\rightarrow$ replace $\bmu^k$ by this solution $\bmu^k \in \U$ \label{ln:SolvePMaxReg}
			\EndWhile
		\end{algorithmic}
	\end{algorithm}

	\begin{figure}
		\centering
		\includegraphics[width=0.85\textwidth]{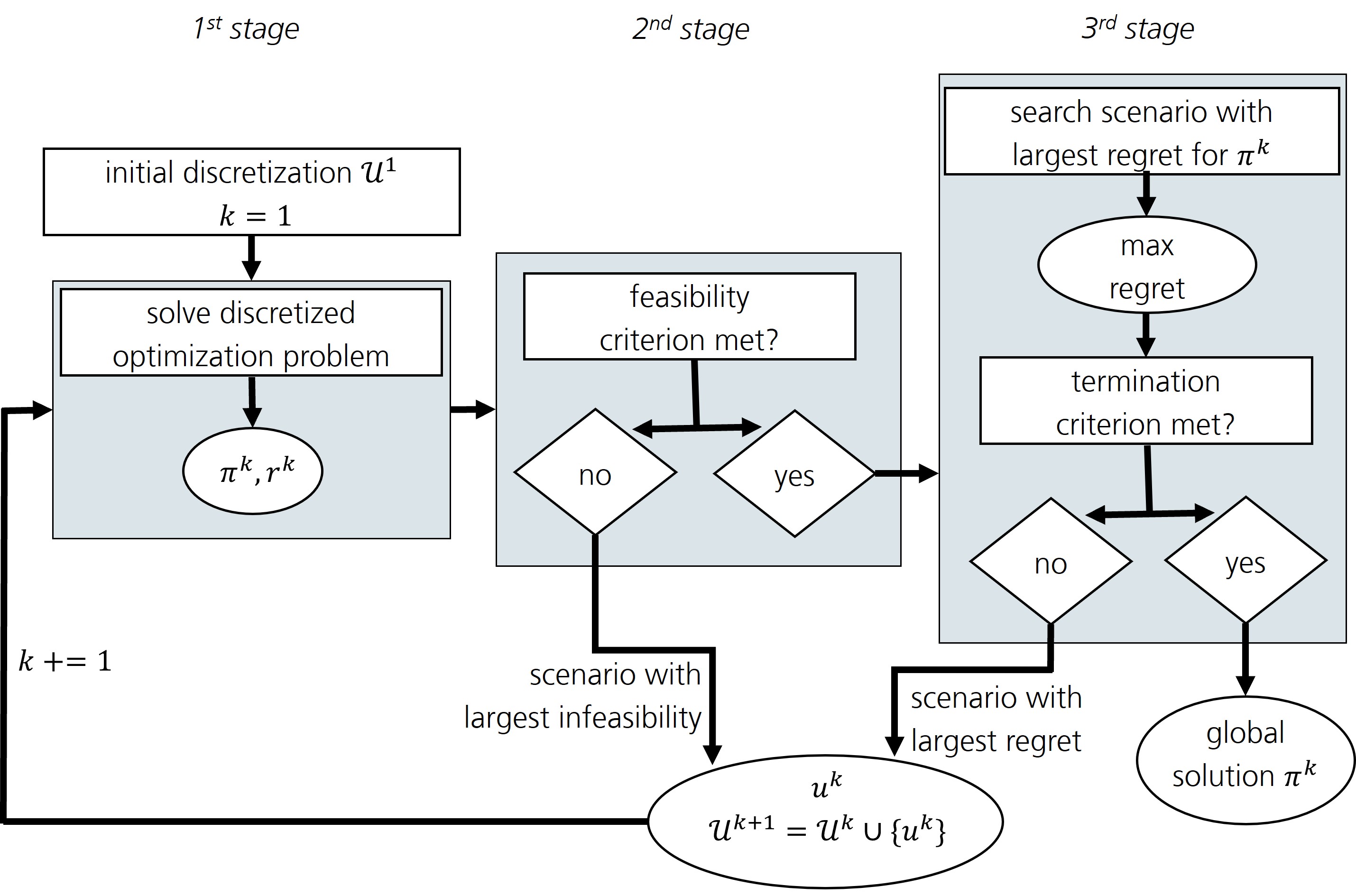}
		\caption{In the \nth{1} stage of Algorithm \ref{alg:3_stage_algo}, the solution of the discretized min-max-regret problem is obtained. The feasibility and optimality of this solution is then checked in the \nth{2} and \nth{3} stage, respectively.
		}
		\label{fig:3_stage_algo}
	\end{figure}
	
     \subsection{Convergence of the algorithm}
     In the following, it is proven that the algorithm converges. Further, an estimate on the distance of an approximate solution generated by Algorithm \ref{alg:3_stage_algo} to the actual solution of \ref{opt_problem:minmaxregret_lower_level_without_decision_rule} in terms of the regret-difference calculated in the $\varepsilon$-termination criterion is given. Before stating and proving these two results, some more notation has to be introduced.

	The feasible set  $F(\bmu)$ of the lower level problem is defined in \eqref{eq:feasible_set_lower_level_problem} and its \emph{optimal value mapping} is defined as
	\begin{equation} \label{eq:definition_optimal_value_function}
		\optimalvaluefunction: \U \rightarrow \R, \quad \bmu \mapsto \min_{\bmx\in F(\bmu)} \objectivefunctioninx(\bmx).
	\end{equation}
	The following lemmas yield some auxiliary results that are used to prove convergence for Algorithm \ref{alg:3_stage_algo}. The first one of these results is the continuity of the optimal value function $\optimalvaluefunction$ of the lower level problem of \ref{opt_problem:minmaxregret_lower_level_without_decision_rule}.
	\begin{lemma} \label{lem:f*_cont}
		For $\setforx \subset \R^\dimx$, let $\objectivefunctioninx\! :\!\setforx \! \rightarrow \! \R$ be continuous and convex, and  
		${F(\bmu) \neq \emptyset}$ for all $\bmu\in \U$.
		Then, the optimal value function $\optimalvaluefunction(\bmu)$ of the lower level problem \ref{opt_problem:subproblem_min_cost} is continuous on $\U$.
	\end{lemma}
	
	\begin{proof}
		The constraints
		\begin{equation}
			g_i(\bmx) \coloneqq (A\bmx - \bmb(\bmu))_i, \quad i=1,..., \numberconstraints,
		\end{equation}
		are linear functions and hence convex and weakly analytic. Therefore, the prerequisites of Theorem~4.3.5 from \cite{Bank.1982} are fulfilled. This theorem then provides the continuity of  $\optimalvaluefunction$ on $\U$.
	\end{proof}
	
	In Algorithm \ref{alg:3_stage_algo}, the first and second stage are repeated until the feasiblity criterion is fulfilled. 
	The next lemma guarantees that this happens after a finite number of iterations. 
	
	\begin{lemma} \label{lem:feasibility_interval}
		Let $\U \subset \R^\dimU$ be interval-shaped, closed and bounded. Then, Algorithm~\ref{alg:3_stage_algo} always enters the \nth{3} stage after a finite number of iterations in the first two stages. 
	\end{lemma}
	This result holds, since the max-infeasibility problem \ref{opt_prob:subproblem_max_infeasibility} is a linear optimization problem, and therefore the solution is an vertex of the intervall-shaped uncertainty set $\U$ \cite{Hochstattler.2017}. There are only finitely many vertices, which is the reason that the feasibility criterion is fulfilled, at last, after all vertices of $\U$ are contained in the discretization $\dot\U$.
	
	The next result uses compactness of the uncertainty set $\U$ to guarantee that for every ${\delta >0}$ there exists an iteration $k$ after which the distance of the iterate $\bmu^k$ to the discretized uncertainty set $\dot{\U^k}$ is smaller than this value $\delta$.
	
	\begin{lemma} \label{lemma:discretization_gets_finer}
	Let $\U\subset \R^\dimU$ be nonempty compact and $\dot{\U}^1\subset \U$ be a nonempty discretization of $\U$. 
	Consider a sequence $\{\bmu^{k}\}_{k\in \N}\subset \U$  and discretizations ${\dot{\U}^{k+1} = \dot{\U}^{k} \cup \{\bmu^k\}},\ k\in \N$, of $\U$ generated by Algorithm \ref{alg:3_stage_algo}. 
	Then the following holds:
	\begin{equation}
		\forall \delta > 0 \ \exists k \in \N: \ \exists \bmu \in \dot{\U}^{k}: \|\bmu^{k} - \bmu\| < \delta.
	\end{equation}
	\end{lemma}
	
	\begin{proof}
	Since $\U \subset \R^\dimU$ is compact, every covering of $\U$ has a finite subcover. Denote by $B_r(\bmu) = \{\bmx \in \R^\dimU | \|\bmu-\bmx\| < r \}$ the open ball around $\bmu$ with radius $r$. Then, for $\delta >0$, the covering 
	\begin{equation}
		\U \subseteq \bigcup_{\bmu \in \U} B_{\nicefrac{\delta}{2}}(\bmu)
	\end{equation}
	has a finite subcover 
	\begin{equation}
		\U \subseteq \bigcup_{\bmv \in \mathcal{V}} B_{\nicefrac{\delta}{2}}(\bmv) \subseteq \bigcup_{\bmu \in \U} B_{\nicefrac{\delta}{2}}(\bmu)
	\end{equation}
	with the finite subset $\mathcal{V} \subset \U$, $|\mathcal{V}| =: I(\delta) < \infty$. 
	
	Assume that for some $\delta >0$ and $k\in \N$, it holds that $\|\bmu^k - \bmu\| \geq \delta$ for all $\bmu \in \dot{\U}^k$. This implies that there exists an $\tilde{\bmv} \in \mathcal{V}$ with 
	\begin{equation}
		\bmu^k\in B_{\nicefrac{\delta}{2}}(\tilde{v}) \quad \text{and} \quad B_{\nicefrac{\delta}{2}}(\tilde{\bmv}) \cap \dot{\U}^k = \emptyset.  
	\end{equation}
	There can be at most $I(\delta)-1$ many such balls without any element of the discretized set. 
	If the same holds for the next discretization $\dot{\U}^{k+1} = \dot{\U}^k \cup \{\bmu^k\}$ and $\bmu^{k+1}$, i.e., 
	${\| \bmu^{k+1} -\bmu \|\geq\delta} \  \forall \bmu \in \dot{\U}^{k+1}$, the same argument as before is used, but now there can be at most $I(\delta)-2$ many such $\frac{\delta}{2}$-balls without any element of the discretized set $\dot{\U}^{k+1}$. This argumentation can be repeated until, after $I(\delta)-1$ iterations, there can be no such $\frac{\delta}{2}$-ball in which no element of the discretized set $\dot{\U}^{k+ I(\delta)-1}$ lies. Hence, there exists an integer $\tilde{k} = k+I(\delta)-1$ for which $\| \bmu^{\tilde{k}} - \bmu\| < \delta$ for some $\bmu \in \dot{\U}^{\tilde{k}}$ is true. 
	\end{proof}

	For the estimate in the next lemma, 
    the modulus of continuity \cite{Suhubi.2003}, which is defined for uniformly continuous functions, is needed. First note, that the function $\objectivefunctioninpi$ is continuous in $\bmpi$ and $\bmu$ because $\objectivefunctioninx$ is continuous in $\bmx$, and the affine linear decision rule  
	\eqref{eq:decision_rule} is continuous in $\bmu$. From Lemma \ref{lem:f*_cont}, it is known that $\optimalvaluefunction$ is continuous on $\U$. And, as continuous functions over a compact set, these functions are also uniformly continuous. Hence, the moduli of continuity exist and are defined as follows:
	\begin{subequations}
	\begin{align}
		\omega(f(\pi, \cdot), \delta) &= \sup_{u,v \in \U} \{|\objectivefunctioninpi(\pi, u) - \objectivefunctioninpi(\pi, v)| \ | \  \|u-v\| \leq \delta \}, \\
		\omega(\optimalvaluefunction(\cdot), \delta) &= \sup_{u,v\in \U} \{|\optimalvaluefunction(u) - \optimalvaluefunction(v)| \  | \  \|u-v \| \leq \delta \}, \\
		\omega(r(\pi, \cdot), \delta) &= \sup_{u,v \in \U} \{|r(\pi, u) - r(\pi, v)| \ | \ \|u-v\| \leq \delta\}. \label{eq:omega_for_regret}
	\end{align}
	\end{subequations}
	The next result relates the distance of the maximal regret of the current iterate over the discretized uncertainty set to its maximal regret over the whole uncertainty set to the parameter $\delta > 0$ that measures the largest possible distance of new discretization candidates from the current discretized uncertainty set.
	\begin{lemma} \label{lemma:distance_bounds_regret}
	Let $\U \subseteq \R^{\dimU}$ be nonempty and compact and let the sequence ${\{\bmu^{k_l}\} \subset \U}$ be generated by solving the max-regret problem $P_{\text{Max-Regret}}(\bmpi^{k_l}, \U)$ in Algorithm \ref{alg:3_stage_algo}. Further, let $\optimalregretvalueinpi^{k_l}$ be the optimal objective value of $P_{\text{Min-Max-Regret}}(\dot{\U}^{k_l})$. Then, it holds
	\begin{equation}
		0\leq \regretfunctioninpi(\bmpi^{k_l}, \bmu^{k_l}) -  \optimalregretvalueinpi^{k_l} \leq \omega(\objectivefunctioninpi(\bmpi^{k_l}, \cdot), \delta) + \omega(\optimalvaluefunction(\cdot), \delta)
	\end{equation}
	\end{lemma}
	
	\begin{proof}
	From Lemma \ref{lem:f*_cont}, it is known that $\optimalvaluefunction(\bmu)$ is continuous on $\U$. 
	And $\objectivefunctioninpi(\bmpi, \bmu)$ is continuous on $\U$, because $\objectivefunctioninx$ is continuous in $\bmx$ and the decision rule \eqref{eq:decision_rule} is continuous in $\bmu$. 
	Therefore, it holds that
	\begin{subequations} \label{eq:omega_to_zero}
		\begin{align}
			\omega(\objectivefunctioninpi(\bmpi, \cdot), \delta) &\rightarrow 0 \text{ as } \delta \rightarrow 0 \\
			\omega (\optimalvaluefunction(\cdot), \delta) &\rightarrow 0 \text{ as } \delta  \rightarrow 0.
		\end{align}
	\end{subequations}
	
	For the given solution $\bmpi^{k_l}$ of the discretized problem $P_{\text{Min-Max-Regret}}(\dot{\U}^{k_l})$, the scenarios $\bmu^{k_l}$ maximize the regret.
	Therefore, it holds that  
	\begin{equation} \label{eq:regret_difference_positive}
		\regretfunctioninpi(\bmpi^{k_l}, \bmu^{k_l}) = \max_{\bmu\in \U} \regretfunctioninpi(\bmpi^{k_l}, \bmu ) \geq \max_{\bmu\in \dot{\U}^{k_l}} \regretfunctioninpi(\bmpi^{k_l}, \bmu ) 
		= \optimalregretvalueinpi^{k_l}.
	\end{equation}

	For $\delta>0$ Lemma \ref{lemma:discretization_gets_finer}, states the existence of $l\in \N$ such that there is an $\bmu \in \dot{\U}^{k_l}$ with $\|\bmu^{k_l} - \bmu\| < \delta$. For this $\bmu\in \dot{\U}^{k_l}$ it holds that
	\begin{equation}
		\regretfunctioninpi(\bmpi^{k_l}, \bmu) \leq \optimalregretvalueinpi^{k_l} = \max_{\bmu\in \dot{\U}^{k_l}} \objectivefunctioninpi(\bmpi^{k_l},\bmu).
		\label{eq:proof_algo_terminates_regret_estimate}
	\end{equation} 
	For the optimal value of the lower level problem, $\optimalvaluefunction(\bmu^{k_l})$, it holds that
	\begin{equation}
		\optimalvaluefunction(\bmu^{k_l}) \leq \objectivefunctioninpi(\bmpi, \bmu^{k_l}) \quad \forall \bmpi \in F(\bmu^{k_l})
	\end{equation}	
	and the regret $r(\bmpi^{k_l})$ is always non-negative. Therefore, it holds that
    \begin{equation}
	\begin{aligned}
		\regretfunctioninpi(\bmpi^{k_l}, \bmu^{k_l}) 
		& = \left|\objectivefunctioninpi(\bmpi^{k_l}, \bmu^{k_l}) - \optimalvaluefunction(\bmu^{k_l})\right| \\
		&= \left| \objectivefunctioninpi(\bmpi^{k_l}, \bmu^{k_l}) - \objectivefunctioninpi(\bmpi^{k_l}, \bmu) + \objectivefunctioninpi(\bmpi^{k_l}, \bmu) - \optimalvaluefunction(\bmu) + \optimalvaluefunction(\bmu) - \optimalvaluefunction(\bmu^{k_l}) \right|  \\
		& \leq  \left| \objectivefunctioninpi(\bmpi^{k_l}, \bmu^{k_l}) - \objectivefunctioninpi(\bmpi^{k_l}, \bmu) \right|  + \left| \objectivefunctioninpi(\bmpi^{k_l}, \bmu) - \optimalvaluefunction(\bmu) \right| \\
        & \quad + \left| \optimalvaluefunction(\bmu)-\optimalvaluefunction(\bmu^{k_l}) \right|. \label{eq:proof_distance_bounds_regret_sum_of_absolute_values}
	\end{aligned}
    \end{equation}
	The last three terms in \eqref{eq:proof_distance_bounds_regret_sum_of_absolute_values} can be further estimated as follows:	
	Since  $\|\bmu^{k_l} - \bmu\| < \delta$, for the first and third summand it holds that
	\begin{align}
		\left| \objectivefunctioninpi(\bmpi^{k_l}, \bmu^{k_l}) - \objectivefunctioninpi(\bmpi^{k_l},\bmu) \right| &\leq \omega(\objectivefunctioninpi(\bmpi, \cdot), \delta),
    \end{align}
    and 
    \begin{align}
		\left| \optimalvaluefunction(\bmu)-\optimalvaluefunction(\bmu^{k_l}) \right| &\leq \omega(\optimalvaluefunction(\cdot), \delta).
	\end{align}
	For the second summand, the absolute value can be omitted and it follows by definition of the regret and \eqref{eq:proof_algo_terminates_regret_estimate}
	
	\begin{equation}
		\left| \objectivefunctioninpi(\bmpi^{k_l}, \bmu) - \optimalvaluefunction(\bmu) \right| =\objectivefunctioninpi(\bmpi^{k_l}, \bmu) - \optimalvaluefunction(\bmu) = \regretfunctioninpi(\bmpi^{k_l}, \bmu) \leq \optimalregretvalueinpi^{k_l}.
	\end{equation}
	Putting everything together, it follows that
	\begin{equation}
		\regretfunctioninpi(\bmpi^{k_l}, \bmu^{k_l}) \leq \omega(\objectivefunctioninpi(\bmpi^{k_l}, \cdot), \delta) + \optimalregretvalueinpi^{k_l} + \omega(\optimalvaluefunction(\cdot), \delta).
		\label{eq:proof_distance_bounds_regret_estimate_2}
	\end{equation}
	From \eqref{eq:regret_difference_positive} and \eqref{eq:proof_distance_bounds_regret_estimate_2} it is obtained that
	\begin{equation}
		0 \leq \regretfunctioninpi(\bmpi^{k_l}, \bmu^{k_l}) - \optimalregretvalueinpi^{k_l} \leq  \omega(\objectivefunctioninpi(\bmpi^{k_l}, \cdot), \delta)+ \omega(\optimalvaluefunction(\cdot), \delta).
		\label{eq:proof_distance_bounds_regret_estimate_3}
	\end{equation}
	\end{proof}
	
	Finally, with the results obtained so far, the global convergence of Algorithm \ref{alg:3_stage_algo} can be shown.
	
	\begin{theorem} \label{thm:convergence_3_stage_algo}
	Let $\setforx \subset \R^\dimx$ be closed and bounded and let $\objectivefunctioninx: \setforx \rightarrow \R$ be continuous and convex. Assume that the uncertainty set $\U \in \R^\dimU$ is closed, bounded and interval-shaped, and that it fulfills $F(\bmu) \neq \emptyset$ for all $\bmu\in \U$. 
	Further, let $\{\bmpi^k\}_{k\in \mathbb{N}}\subseteq \setforpi$ be a sequence generated by Algorithm \ref{alg:3_stage_algo}. Then $\{\bmpi^k\}_{k\in \mathbb{N}}$ has an accumulation point $\hat{\bmpi} \in \setforpi$ and every accumulation point $\hat{\bmpi}$ of $\{\bmpi^k\}_{k\in \mathbb{N}}$ 
	is a global solution of \ref{opt_problem:minmaxregret_lower_level_without_decision_rule}.
	\end{theorem}
		
	\begin{proof}
	In this proof, the following steps are shown:
	\begin{enumerate}
		\item  $\{\bmpi^k\}_{k\in \mathbb{N}}$ has an accumulation point.
		\item Every accumulation point is feasible.
		\item The algorithm terminates after a finite number of iterations.
		\item When the termination criterion is met, say after $k$ iterations, then $\bmpi^k$ approximates a global solution of \ref{opt_problem:minmaxregret_lower_level_without_decision_rule}.
	\end{enumerate}

	First it is shown that the sequence  $\left\{\bmpi^k\right\}_{k\in \mathbb{N}} \subset \setforpi \subset \R^{\dimpi}$, generated by Algorithm~\ref{alg:3_stage_algo}, has an accumulation point. Since $\setforpi \subset \R^n$ is bounded and closed, every sequence in $\setforpi$ has a convergent subsequence. This implies that  $\{\bmpi^k\}_{k\in \mathbb{N}}$ has an accumulation point. 
	
	The next step is to show that every accumulation point is feasible for \ref{opt_problem:minmaxregret_lower_level_without_decision_rule}. Let $\hat{\bmpi}$ be an accumulation point of  $\{\bmpi^k\}_{k\in \mathbb{N}}$ and let $\{\bmu^k\}_{k\in \N}$ be generated by Algorithm~\ref{alg:3_stage_algo}. Define 
	\begin{equation}
		g_j(\bmpi, \bmu)~\coloneqq~(\Aforpi(\bmu) \bmpi - b(\bmu))_j 
	\end{equation}
	as the $j$-th constraint, $j=1,..., \numberconstraints$.
	For every $k' \geq k+1$ it holds that $ \bmu^k \in \dot{\U}^{k'} $ and therefore
	\begin{equation}
		g_j(\bmpi^{k'}, \bmu^k) \leq 0 \quad \forall \ 1 \leq j \leq \numberconstraints.
	\end{equation}	 
	Since $\mathcal{U}$ is compact, 
	we can choose convergent subsequences $\{\bmpi^{k_l}\}_{l\in \mathbb{N}}$ and $\{\bmu^{k_l}\}_{l\in \mathbb{N}}$ of sequences generated by Algorithm \ref{alg:3_stage_algo} with
	\begin{equation}
		\lim_{l\rightarrow \infty} \bmpi^{k_l} = \hat{\bmpi} \quad \text{and} \quad \lim_{l \rightarrow \infty} \bmu^{k_l} = \hat{\bmu}
	\end{equation}
	for some $\hat{\bmu} \in \U$.
	Let $\bmu \in \U$ be arbitrary. It follows that 
	\begin{align}
		\max_{j\in \{1,...,\numberconstraints\}} \left\{g_j(\hat{\bmpi}, \bmu)\right\} &= \max_{j\in \{1,...,\numberconstraints\}} \left\{ g_j \left(\lim_{l\rightarrow \infty} \bmpi^{k_l}, \bmu \right) \right\} \\
		&= \lim_{l\rightarrow \infty} \max_{j\in \{1,...,\numberconstraints\}} \left\{ g_j(\bmpi^{k_l}, \bmu) \right\}\\
		& \leq \lim_{l\rightarrow \infty} \max_{j\in \{1,...,\numberconstraints\}} \left\{ \max \left\{ g_j(\bmpi^{k_l}, \bmu^{k_l}), 0 \right\} \right\} \label{eq:proof_feasibility_accumulation_point_leq_estimate}\\
		&=  \max_{j\in \{1,...,\numberconstraints\}} \left\{ \max \left\{ g_j(\lim_{l\rightarrow \infty} \bmpi^{k_l}, \lim_{l\rightarrow \infty} \bmu^{k_l}), 0 \right\} \right\} \\
		&= \max_{j\in \{1,...,\numberconstraints\}} \left\{ \max \left\{ g_j(\lim_{l\rightarrow \infty} \bmpi^{k_{l+1}}, \lim_{l\rightarrow \infty} \bmu^{k_l}), 0 \right\} \right\} \label{eq:proof_feasibility_accumulation_point_index_shift}\\
		&= \lim_{l\rightarrow \infty} \max_{j\in \{1,...,\numberconstraints\}} \left\{ \max \left\{ g_j(\bmpi^{k_{l+1}}, \bmu^{k_l}), 0 \right\} \right\}  \leq 0 \label{eq:proof_feasibility_accumulation_point_higher_iteration_feasible}
	\end{align}
	The estimate in \eqref{eq:proof_feasibility_accumulation_point_leq_estimate} holds because there are two cases that can occur. Either $\bmu^{k_l}$ maximizes $\max_{\bmu\in \U} \ \max_{j\in \{1,...,\numberconstraints\}}g_j(\bmx^{k_l}, \bmu)$ in which case the estimate holds. Or, if $\bmu^{k_l}$ is generated by solving $P_{\text{Max-Regret}}(\bmpi^{k_l}, \U)$, the solution $\bmpi^{k_l}$ is feasible for every $\bmu\in \U$ in which case ${g_j(\bmpi^{k_l}, \bmu) \leq 0}$ holds for every $j=1,...,\numberconstraints$ and every $\bmu\in \U$. Then, in \eqref{eq:proof_feasibility_accumulation_point_index_shift}, an index-shift is made, and in the next step it is used that $\bmpi^{k_{l+1}}$ is feasible for $\bmu^{k_l}$ which implies $g_j(\bmpi^{k_{l+1}}, \bmu^{k_l}) \leq 0$ holds for every $j=1,...,\numberconstraints$.
	
	This proves that the accumulation point $\hat{\bmpi}$ fulfills the constraints of \ref{opt_problem:minmaxregret_lower_level_without_decision_rule}.
	
	In the third step of this proof, it is shown that the algorithm terminates. 
	Consider scenarios $\{\bmu^{k_l}\}\subset \U$ that are added in the third stage of Algorithm \ref{alg:3_stage_algo}, i.e., these are solutions of $P_{\text{Max-Regret}}(\bmpi^{k_l},\U)$. By Lemma \ref{lemma:distance_bounds_regret} it follows that
	\begin{equation}
		0\leq \regretfunctioninpi(\bmpi^{k_l}, \bmu^{k_l}) -  \optimalregretvalueinpi^{k_l} \leq \omega(\objectivefunctioninpi(\bmpi^{k_l}, \cdot), \delta) + \omega(\optimalvaluefunction(\cdot), \delta).
	\end{equation}
	For $l\rightarrow \infty$, it holds that $\delta \rightarrow 0$, and using \eqref{eq:omega_to_zero} and \eqref{eq:proof_distance_bounds_regret_estimate_3} one obtains
	\begin{equation}
		\regretfunctioninpi(\bmpi^{k_l},\bmu^{k_l}) - \optimalregretvalueinpi^{k_l} \rightarrow 0.
	\end{equation}
	Thus, the termination condition $ \regretfunctioninpi(\bmpi^{k_l},\bmv^{k_l}) - \optimalregretvalueinpi^{k_l} < \varepsilon$ is met after a finite number of iterations and the algorithm terminates.
	
	It remains to show that, once the termination criterion is met, say after the $k$-th iteration, $\bmpi^k$ approximates a global solution of \ref{opt_problem:minmaxregret_lower_level_without_decision_rule}.
	Denote the global 
	optimal objective value of \ref{opt_problem:minmaxregret_lower_level_without_decision_rule} by $\optimalregretvalueinpi^*$ and its feasible set by 
	\begin{equation}
		F^* \coloneqq \{\bmpi \in \setforpi | \Aforpi (\bmu)\bmpi \leq b(\bmu) \quad \forall \bmu\in \U\}.
	\end{equation} 
	Accordingly, denote the objective value of the discretized problem $P_{\text{Min Max Regret}}(\dot{\mathcal{U}}^k)$ by $\optimalregretvalueinpi^k$ and its feasible set by 
	\begin{equation}
		F^k \coloneqq \{\bmpi \in \setforpi | \Aforpi (u)\bmpi \leq b(u) \quad \forall u\in \dot{\U}^k \}.
	\end{equation} 
	With every iteration, the discretization of the uncertainty set gets larger and it holds
	\begin{equation}
		\dot{\U}^1 \subset ... \subset \dot{\U}^k \subset \dot{\U}^{k+1} \subset ... \subset \U.
	\end{equation}  Hence,
	\begin{equation}
		F^* \subseteq ... \subseteq F^{k+1} \subseteq F^k \subseteq ... \subseteq F^1
	\end{equation}
	and for the objective values it holds
	\begin{equation}
		\optimalregretvalueinpi^{1} \leq ... \leq \optimalregretvalueinpi^{k} \leq \optimalregretvalueinpi^{k+1} \leq ... \leq \optimalregretvalueinpi^\ast. 
		\label{eq:proof_algo_global_solution_optimal_regret_value_order}
	\end{equation}
	Let $\{\bmu^{k_l}\}_{l\in \N}$ be the subsequence of $\{\bmu^k\}_{k\in \N}$ that is generated by maximizing the regret for a given solution $\bmpi^{k_l}$ of the discretized problem, i.e., 
	\begin{equation}
		(\bmu^{k_l}, \bmx^*) = \underset{\bmu\in \U, \bmx^*(\bmu) \in \setforx}{\arg \max} \left\{\left.\regretfunctioninpi(\bmpi^{k_l}, \bmu) = \objectivefunctioninpi(\bmpi^{k_l}, \bmu) - \objectivefunctioninx(\bmx^*(\bmu))\right| A\bmx^*(\bmu)\leq \bmb(\bmu) \right\}
	\end{equation} (Line~\ref{ln:SolvePMaxReg} in Algorithm \ref{alg:3_stage_algo}). Such a sequence exists since, by Lemma \ref{lem:feasibility_interval}, after a finite number of iterations of the feasibility problem, its solution is feasible for every $\bmu\in \U$.
	It then holds that
	\begin{equation}
		\optimalregretvalueinpi^* = \min_{\bmpi \in \setforpi} \max_{\bmu\in \U} \regretfunctioninpi(\bmpi, \bmu) \leq \max_{\bmu\in \U} \regretfunctioninpi(\bmpi^{k_l},\bmu) = \regretfunctioninpi(\bmpi^{k_l}, \bmu^{k_l}).
	\end{equation}
	With this and \eqref{eq:proof_algo_global_solution_optimal_regret_value_order} it follows
	\begin{equation}
		\optimalregretvalueinpi^{k_l} \leq \optimalregretvalueinpi^* \leq \regretfunctioninpi(\bmpi^{k_l}, \bmu^{k_l}).
		\label{eq:estimate_regret_values}
	\end{equation}
	If, for some given $\varepsilon>0$, the $\varepsilon$-termination condition of the algorithm is met, i.e.,
	\begin{equation}
		\regretfunctioninpi(\bmpi^{k_l}, \bmu^{k_l}) - \optimalregretvalueinpi^{k_l} < \varepsilon,
	\end{equation}
	the difference of the global optimal objective value $\optimalregretvalueinpi^*$ to the objective value $\optimalregretvalueinpi^{k_l}$ of the solution $\bmpi^{k_l}$ of the discretized problem can be estimated by
	$$|\optimalregretvalueinpi^\ast - \optimalregretvalueinpi^{k_l}| = \optimalregretvalueinpi^\ast - \optimalregretvalueinpi^{k_l}  \leq \regretfunctioninpi(\bmpi^{k_l}, \bmu^{k_l}) - \optimalregretvalueinpi^{k_l} < \varepsilon$$
	
	By choosing $\varepsilon$ small enough it can be ensured that the algorithm stops arbitrarily close to a global solution of \ref{opt_problem:minmaxregret_lower_level_without_decision_rule}.
	\end{proof}
    
	\begin{remark}
	The sequence $\{\regretfunctioninpi^k\}_{k\in \N}$ generated by the first stage of Algorithm \ref{alg:3_stage_algo} (Line \ref{ln:SolvePMinMaxRegretDiscretized}) is a lower bound on the optimal objective value $\regretfunctioninpi^*$ of \ref{opt_problem:minmaxregret_lower_level_without_decision_rule} and is monotonically increasing. 
	
	The sequence $\{r(\bmpi^k, \bmu^k)\}_{k\in \N}$ with $\bmu^k$ from the third stage of Algorithm \ref{alg:3_stage_algo} (Line \ref{ln:SolvePMaxReg}) is an upper bound on the optimal objective value $\regretfunctioninpi^*$ of \ref{opt_problem:minmaxregret_lower_level_without_decision_rule}, but, in general, it is not monotone.  
	The best upper bound in iteration $k$ of $\regretfunctioninpi^*$ is thus given by \begin{equation}
		\min_{1\leq l \leq k} r(\bmpi^l, \bmu^l)
	\end{equation}
	with $\bmu^l$ being generated in the third stage of the algorithm. 
	\end{remark}
	
	In the following it is shown how the error of a current solution $\bmpi^k$ that fulfills the feasibility criterion can be estimated from the regret-difference calculated in the $\varepsilon$-termination criterion.
	For the following theorem, maximal regret values for a parameter $\bmpi$ over different uncertainty sets must be compared. Therefore, the following notation for the maximal regret over the whole uncertainty set $\U$, resp. a discretization $\dot{\U}^k$ of the uncertainty set, is introduced.
	
	Denote by $R(\bmpi)$ the optimal objective value of problem \ref{opt_prob:subproblem_max_regret} and by $\dot{R}^k(\bmpi)$ the optimal objective value of the following problem
	\begin{equation}
	\begin{aligned}
		\max_{u\in \dot{\U}^k, \bmx^*(\bmu) \in X} \quad & \objectivefunctioninpi(\bmpi, \bmu) - \objectivefunctioninx(\bmx^*(\bmu)) \\
		\text{s.t. } \quad & A\bmx^*(\bmu) \leq b(\bmu),
	\end{aligned}
	\end{equation}
	where only a discrete subset $\dot{\U}^k \subset \U$ is considered.
	
	\begin{definition}[{Solution of order p \cite{Still.2001}}]
	A point $\bmpi^* \in F$ is a solution of order $p=1,2$ of problem \ref{opt_problem:minmaxregret_lower_level_without_decision_rule}
	if there exists a constant $c >0$ and a neighborhood $V$ of $\bmpi^*$ such that 
	\begin{equation}
		c \| \bmpi - \bmpi^* \|^p \leq R(\bmpi) - R(\bmpi^*) 
		\label{eq:solution_of_order_p}
	\end{equation}
	holds for all $\bmpi \in V\cap F$. 
	\end{definition}
	
	\begin{theorem}
	Assume that $\bmpi^*$ is a global minimizer of \ref{opt_problem:minmaxregret_lower_level_without_decision_rule} of order $p \in \{1,2\}$. Then, the following holds
	\begin{enumerate}
		\item There exists a constant $c>0$ such that for every $\bmpi^k$ which fulfills the feasibility criterion, it holds that
		\begin{equation}
			\| \bmpi^k - \bmpi^*\| \leq \frac{1}{c} \sqrt[p]{\delta^k}
		\end{equation}
        with $\delta^k \coloneqq R(\bmpi^k) - \dot{R}^k(\bmpi^k)$.
		\item For a given $\delta>0$, there exists $k' \in \N$, such that the estimate
		\begin{equation}
			\| \bmpi^k - \bmpi^*\| \leq \sqrt[p]{\delta}
		\end{equation}
		is true for every $k \geq k'$.\\
	\end{enumerate}
	\end{theorem}
	
	\begin{proof}
	For the proof of the first statement, the estimate \eqref{eq:estimate_regret_values} on for the different regrets is used. For easier notation, the index of the subsequence will be omitted in the following. But it should be kept in mind that only $\bmpi^k$ that fulfill the feasibility criterion are considered. 
	Equation \eqref{eq:estimate_regret_values} can be written as  
	\begin{equation}
		\dot{R}^k(\bmpi^k) \leq R(\bmpi^*) \leq R(\bmpi^k).
		\label{eq:estimate_regret_values_capital_R}
	\end{equation}
	For checking the $\varepsilon$-termination criterion, the current difference in the maximal regret over the discretized uncertainty set $\dot{\U}^k$ and the maximal regret over the whole uncertainty set $\U$, 
	\begin{equation}
		\delta^k = R(\bmpi^k) - \dot{R}^k(\bmpi^k) \geq 0,
	\end{equation}
	is calculated.
	With this and Equation \eqref{eq:estimate_regret_values_capital_R}, it follows that
	\begin{equation}
		R(\bmpi^k) - \underbrace{R(\bmpi^*)}_{\geq \dot{R}^k(\bmpi^k)} \leq R(\bmpi^k) - \dot{R}^k(\bmpi^k) = \delta^k 
	\end{equation}
	holds true.
	By assumption, $\bmpi^*$ is a solution of order $p$, i.e., there exists a constant $c > 0$ and a neighborhood of $\bmpi^*$ with
	\begin{equation}
		c \| \bmpi - \bmpi^*\|^p \leq R(\bmpi) - R(\bmpi^*) 
	\end{equation}
	for every $\bmpi$ in this neighborhood of $\bmpi^*$.
	This estimate is true for $\bmpi = \bmpi^k$ with $k$ large enough. Therefore, it follows
	\begin{equation}
		\|\bmpi^k - \bmpi^*\| \leq \frac{1}{c} \sqrt[p]{R(\bmpi^k) - R(\bmpi^*)} \leq \frac{1}{c} \sqrt[p]{\delta^k}.
	\end{equation}
	
	To prove the second statement, let $\delta>0$. 
	After at most $k_0 \coloneqq 2^\dimU$ iterations, the current solution $\bmpi^k$ of the discretized problem \ref{opt_prob:subproblem_min_max_regret} is feasible for every scenario $u\in \U$ (cf. Lemma \ref{lem:feasibility_interval}).
	From the proof of the first statement of this theorem, it is known that $\|\bmpi^k - \bmpi^* \| \leq \sqrt[p]{\delta}$ holds true if $R(\bmpi^k) - \dot{R}^k(\bmpi^k) \leq c^p \delta$, where $c$ is the constant from Estimate \eqref{eq:solution_of_order_p}.
	
	The regret function is continuous in $\bmu$. Therefore, 
	\begin{equation}
		\omega(r(\bmpi, \cdot), \alpha) \rightarrow 0, \quad \text{as } \alpha \rightarrow 0,
	\end{equation}
	where $\omega(r(\bmpi, \cdot), \alpha)$ is defined by \eqref{eq:omega_for_regret}.
	Therefore, for $\delta>0$, there exists an $\alpha' >0$ with
	\begin{equation}
		\omega(r(\bmpi, \cdot), \alpha) < c^p \delta 
	\end{equation}
	for every $\alpha \leq \alpha'$. 
	
	Since $\U$ is compact, Lemma \ref{lemma:discretization_gets_finer} can be applied. This implies that for any $\alpha >0$, there exists an iteration $k\in \N$ such that for $\bmu^k \in \U \setminus \dot{\U}^k$ there exists an $\bmu \in \dot{\U}^k$ with $\| \bmu^k - \bmu \| < \alpha$. From the proof of Lemma \ref{lemma:discretization_gets_finer} it can be seen that this can take at most $I(\alpha)-1$ iterations, where $I(\alpha)$ is the number of elements in the finite subcover. 
	
	Let $\bmpi^k \in \setforpi$ be an iterate that fulfills the feasibility criterion and let $\bmu^k\in \U$ be the solution of the max-regret problem $P_{\text{Max Regret}}(\bmpi^k, \U)$. 
	Further, let $u\in \dot{\U}^k$ be a scenario from the discretized uncertainty set with $\|\bmu^k - \bmu\| < \alpha$. Then, the following estimate is true:
	\begin{align}
		R(\bmpi^k) &= |r(\bmpi^k, \bmu^k)| \label{eq:estimate_R_and_omega_start}\\
		&= |r(\bmpi^k, \bmu^k) - r(\bmpi^k, \bmu) + r(\bmpi^k, \bmu) | \\
		& \leq |r(\bmpi^k, \bmu^k) - r(\bmpi^k, \bmu)|+ |r(\bmpi^k, \bmu)| \label{eq:estimate_R_and_omega}\\ 
		& \leq \omega(r(\bmpi^k, \cdot), \alpha) + \dot{R}^k(\bmpi^k). \label{eq:estimate_R_and_omega_end}
	\end{align}
	Since $\|\bmu^k - \bmu\| < \alpha$, the first summand of \eqref{eq:estimate_R_and_omega} cannot be larger than $\omega(r(\bmpi^k, \cdot), \alpha)$. And, because of $\bmu\in \dot{\U}^k$, the following estimate is true for the second summand.
	\begin{equation}
		|r(\bmpi^k, \bmu)| = r(\bmpi^k, \bmu) \leq \max_{\bmu'\in \dot{\U}^k} r(\bmpi^k, \bmu') = \dot{R}^k(\bmpi^k)
	\end{equation}
	Rearranging the inequalities \eqref{eq:estimate_R_and_omega_start} -- \eqref{eq:estimate_R_and_omega_end} gives
	\begin{equation}
		R(\bmpi^k) - \dot{R}^k(\bmpi^k) \leq \omega(r(\bmpi^k, \cdot), \alpha).
	\end{equation}
	Putting everything together, it follows that after at most $k' = k_0 + I(\alpha)-1$ iterations, there exists an $\bmu\in \dot{\U}^k$ for the iterate $\bmu^{k}$  with
	\begin{equation}
		\|\bmu^{k'} - \bmu\| < \alpha
	\end{equation}
	for every $k \geq k'$.
	It further follows that
	\begin{equation}
		R(\bmpi^k)-\dot{R}^k (\bmpi^k) \leq \omega(r(\bmpi^k, \cdot), \alpha) < c^p \delta,
	\end{equation}
	which implies
	\begin{equation}
		\| \bmpi^k - \bmpi^*\| \leq \sqrt[p]{\delta}.
	\end{equation}
	\end{proof}
	
    \subsection{Discussion of the algorithm with regard to the interval-shaped uncertainty set}

    In this article, only interval-shaped uncertainty sets are considered. These have some special properties that might be exploited to reduce some of the steps of Algorithm \ref{alg:3_stage_algo}. In the following, it is discussed which steps could be omitted and what the consequences thereof would be.
    
    Due to linearity, a solution $\bmpi$ of \ref{opt_problem:minmaxregret_lower_level_without_decision_rule} fulfills the feasibility criterion, i.e., the semi-infinite constraints of \ref{opt_problem:minmaxregret_lower_level_without_decision_rule}, if $\bmpi$ is feasible for every extremal scenario,
    \begin{equation} \label{eq:constraints_for_U_extr}
	\Aforpi (\bmu) \bmpi \leq \bmb(\bmu) \quad \forall \bmu\in \extremalpointsofU.
	\end{equation}
    Hence, one could replace the semi-infinite constraint in  \ref{opt_problem:minmaxregret_lower_level_without_decision_rule} by finitely many constraints resulting from the  extremal scenarios, i.e., \eqref{eq:constraints_for_U_extr}. 
     
	As a consequence, one could start Algorithm \ref{alg:3_stage_algo} with an initial discretization that contains all extremal scenarios, $\extremalpointsofU \subseteq \dot{\U}^1$. Hence, the feasibility criterion would be always fulfilled and solving the max-infeasibility problem $P_{\text{Max Infeasibility}}$ could be omitted in Algorithm \ref{alg:3_stage_algo}. This implies that Lines \ref{ln:SolvePFeas}~--~\ref{ln:endif} would not be needed in this case. Furthermore, Line \ref{ln:SolvePLowerLevel} could also be omitted, because only scenarios with maximal regret are added to the discretization and for these scenarios, the perfect-information solution can be obtained in Line \ref{ln:SolvePMaxReg}.
	In \cite{Shimizu.1980}, the authors propose an adaptive algorithm for a min-max optimization problem with finitely many constraints that is equivalent to Algorithm \ref{alg:3_stage_algo} without Lines \ref{ln:SolvePLowerLevel} and \ref{ln:SolvePFeas}~--~\ref{ln:endif}. However, their algorithm has different prerequisites for convergence and some statements in their convergence proof are not thoroughly proven.
 
	Although some of the steps of Algorithm \ref{alg:3_stage_algo} could be omitted, it is not advisable to make the initial discretization in Algorithm \ref{alg:3_stage_algo} so large that it contains every extremal scenario. The reason for this is that in this case, the lower level problem \ref{opt_problem:subproblem_min_cost} has to be solved at least for every points in $\extremalpointsofU$ which are $2^{\dimU}$ scenarios. In Section \ref{section:performance_algo_depending_on_initial_discretization} it can be seen that this can lead to a significantly larger computation time than choosing smaller initial discretizations and go through all three stages of Algorithm \ref{alg:3_stage_algo}.

    Independent of the choice of the initial discretization of the uncertainty set, the set from where scenarios are added to the discretization in Algorithm \ref{alg:3_stage_algo} should not be restricted to $\extremalpointsofU$: While the scenario with the largest infeasibility is always one of the extremal points of $\U$, the scenario with the largest regret can be one of the scenarios in the interior of $\U$. For the convergence of Algorithm \ref{alg:3_stage_algo}, it is therefore required that the complete uncertainty set $\U$ is taken into account for \ref{opt_prob:subproblem_max_regret}.

	\section{Results} \label{section:results}
In this section, the performance of the algorithm and its different stages is analyzed depending on the choice of the initial discretization of the uncertainty set. Then, the combination of min-max-regret robustness with an affine linear decision rule is compared to adjustable robustness with the same decision rule. For this purpose, two examples are shown and compared with respect to the evaluation measures maximal cost, nominal cost, and maximal regret.

     The following optimization problem is considered in this chapter:
    \begin{equation}\label{opt_problem:results}
    \begin{aligned}
	\min_{\substack{x(p,t)\in \R \\ p=1,...,P, \ t=1,..,T}} \quad & \sum_{t=1}^T e(t) \left( \sum_{p=1}^{P} c_{p,2} x(p,t)^2 + c_{p,1} x(p,t) + c_{p,0} \right) \\
	\text{s.t. } \quad & 0 \leq x(p, t) \leq Q(p), &&\hspace{-2.6cm} p=1,..., P, \ t=1,...,T,\\
	& h^{min} \leq h(t) \leq h^{max}, &&\hspace{-0.5cm} t=1,...,T,\\
	& h^{min, T} \leq h(T), \\
	& h(t) = h(t-1) + \frac{1}{A} \left(\sum_{p=1}^P {x(p,t)} - u(t)\right),  && \hspace{-0.5cm}t=1,..., T, \\
	& u^{min}(t) \leq u(t) \leq u^{max}(t) &&\hspace{-0.5cm} t=1,...,T,
    \end{aligned}
    \end{equation}
    where the parameters $u(t)$ are uncertain. Such problems arise, for example in drinking water supply in order to find an optimal steering of water pumps under uncertain demand. Another application is an inventory problem, which is similarly modeled in \cite{BenTal.2009}. 
    The decision variables $x(p,t), p=1,...,P, t=1,..,T$ are replaced by affine linear decision rules 
    \begin{equation}
        x(p,t) = \pi^0(p,t) + \sum_{r \in I_t} \pi^r(p,t) u(r), \quad p=1,...,P, \ t=1,..,T.
    \end{equation}
    with $|\pi^r(p, t)| \leq N$ for some $N \in \R^+$. 
    The set $I_t \subseteq \{1,...,T\}, \ t=1,...,T$ is called \emph{informations basis} and contains the indices of the uncertain parameter $\bmu$, to which the decisions $x(p, t), p=1,...,P$ can be adjusted. It is assumed that the information basis is given by $I_t = \{1,...,t-\kappa\}, \ t=1,...,T$. 

	\subsection{Performance depending on initial discretization.} \label{section:performance_algo_depending_on_initial_discretization}
	As mentioned above, the second stage of Algorithm \ref{alg:3_stage_algo} can be skipped if every extremal scenario is contained in the initial discretization, i.e., $\extremalpointsofU \subseteq \dot{\U}^1$. The question that arises now is, whether the initial discretization should be chosen such that it fulfills this condition. This would imply that the initial discretization contains at least $2^\dimU$ scenarios. In this case, the lower level problem \ref{opt_problem:subproblem_min_cost} has to be solved for every one of these scenarios in the initialization of Algorithm \ref{alg:3_stage_algo}. In return, the second stage of the algorithm could be skipped. Note that this corresponds to the algorithm proposed in \cite{Shimizu.1980} (with different assumptions on the objective function and constraints), as explained before. Otherwise, one could choose a smaller initial discretization and solve less problems in the initialization of the algorithm but include the second stage of Algorithm \ref{alg:3_stage_algo}. In the following, Problem \eqref{opt_problem:results} is solved for different sizes of the initial discretization of the uncertainty set $\U$ and the computational times are compared. Problem \eqref{opt_problem:results} is considered with the parameters given in Table \ref{tab:parameter_settings_example_4_1}, where $\varepsilon$ is the value used in the $\varepsilon$-termination criterion (cf. Def. \ref{def:termination_criterion}).

\begin{table}[]
    \centering
    \begin{tabular}{|l|llllll|}
		\hline
		$P$ &  \multicolumn{6}{l|}{2} \\
		$T$ & \multicolumn{6}{l|}{12}\\
		$\kappa$  & \multicolumn{6}{l|}{1} \\
		$N$ & \multicolumn{6}{l|}{500}   \\
        $\varepsilon$ & \multicolumn{6}{l|}{0.001} \\
		$A$& \multicolumn{6}{l|}{2200} \\
		$h^{min}$& \multicolumn{6}{l|}{4.2} \\
		 $h^{min,T}$& \multicolumn{6}{l|}{5}\\
	 	$h(0)$ & \multicolumn{6}{l|}{6}  \\
	 	 $h^{max}$ & \multicolumn{6}{l|}{10} \\
	 	 \hline \hline 
	 	 %\multicolumn{13}{|l|}{parameters for $t$} \\
	 	 $t$ & 1 & 2 & 3 & 4 & 5 & 6 \\ \hline 
	 	 $e(t)$ & 0.5 & 0.5&0.3&0.6& 1.2& 1.1 \\ 
	 	 $u^{min}(t)$ & 429.65 & 758.83 &  918.26& 1377.55 & 1616.99 &
 1071.87 \\ %
	 	 $u^{max}(t)$ &  474.87 & 855.70& 1099.90 & 1683.67 & 2057.98&
 1392.20   \\ \hline 
    $t$ &7 & 8 & 9 & 10 & 11 & 12 \\ \hline
    $e(t)$ &1.0& 0.7& 0.6& 0.9& 1.1& 1.2 \\
    $u^{min}(t)$ &1483.87 & 2002.045 & 1304.35 & 1527.15 &
 1099.80 &  655.59 \\
 $u^{max}(t)$ & 1741.94 & 2496.93 & 1764.71 & 1943.65& 
 1344.20 &  754.28  \\ \hline  \hline 
	 	 %\multicolumn{3}{|l|}{parameters for $p$:} \\
	 	 $p$ & \multicolumn{2}{l}{1} & \multicolumn{4}{l|}{2} \\ \hline
	 	 $c_{p,2}$ & \multicolumn{2}{l}{0.000404} & \multicolumn{4}{l|}{0.0003} \\
	 	 $c_{p,1}$ & \multicolumn{2}{l}{$-0.07334$} & \multicolumn{4}{l|}{$-0.06$} \\
	 	 $c_{p,0}$ & \multicolumn{2}{l}{27.78} & \multicolumn{4}{l|}{20} \\
	 	 $Q(p)$ & \multicolumn{2}{l}{1800} & \multicolumn{4}{l|}{1300} \\
	 	 \hline
\end{tabular}
    \caption{Parameters of Problem \eqref{opt_problem:results} considered in the example in Section \ref{section:performance_algo_depending_on_initial_discretization}.}
    \label{tab:parameter_settings_example_4_1}
\end{table}

	This problem was implemented in Pyomo \cite{Bynum.2021, Hart.2011} and solved 
    using Algorithm \ref{alg:3_stage_algo}. The subproblems of the \nth{1} stage, \ref{opt_prob:subproblem_min_max_regret} and \ref{opt_problem:subproblem_min_cost}, were solved using KNITRO \cite{Byrd2006}, the subproblems of the \nth{2} stage, \ref{opt_prob:subproblem_max_infeasibility}  and \ref{opt_prob:subproblem_max_regret}, were solved using Gurobi \cite{gurobi}. Both solvers were used with default settings.
    Without the decision rule, this problem would have 24 decision parameters. With the affine linear decision rule and the chosen information basis, this number is increased to 156 decision parameters. 
 
    The initial discretizations were chosen as  subsets of the set of all extremal scenarios, $\dot{\U}^1 \subseteq \extremalpointsofU$. The size of $\dot{\U}^1$ is measured as the portion of the number of elements in the initial discretization to the number of all extremal scenarios, i.e. $\frac{|\dot{\U}^1|}{|\extremalpointsofU|}$. For each size, the problem was solved ten times with the initial discretization being set by randomly selecting the appropriate number of elements from $\extremalpointsofU$. 
	\begin{figure}
		\centering
		\includegraphics[width=0.85\textwidth]{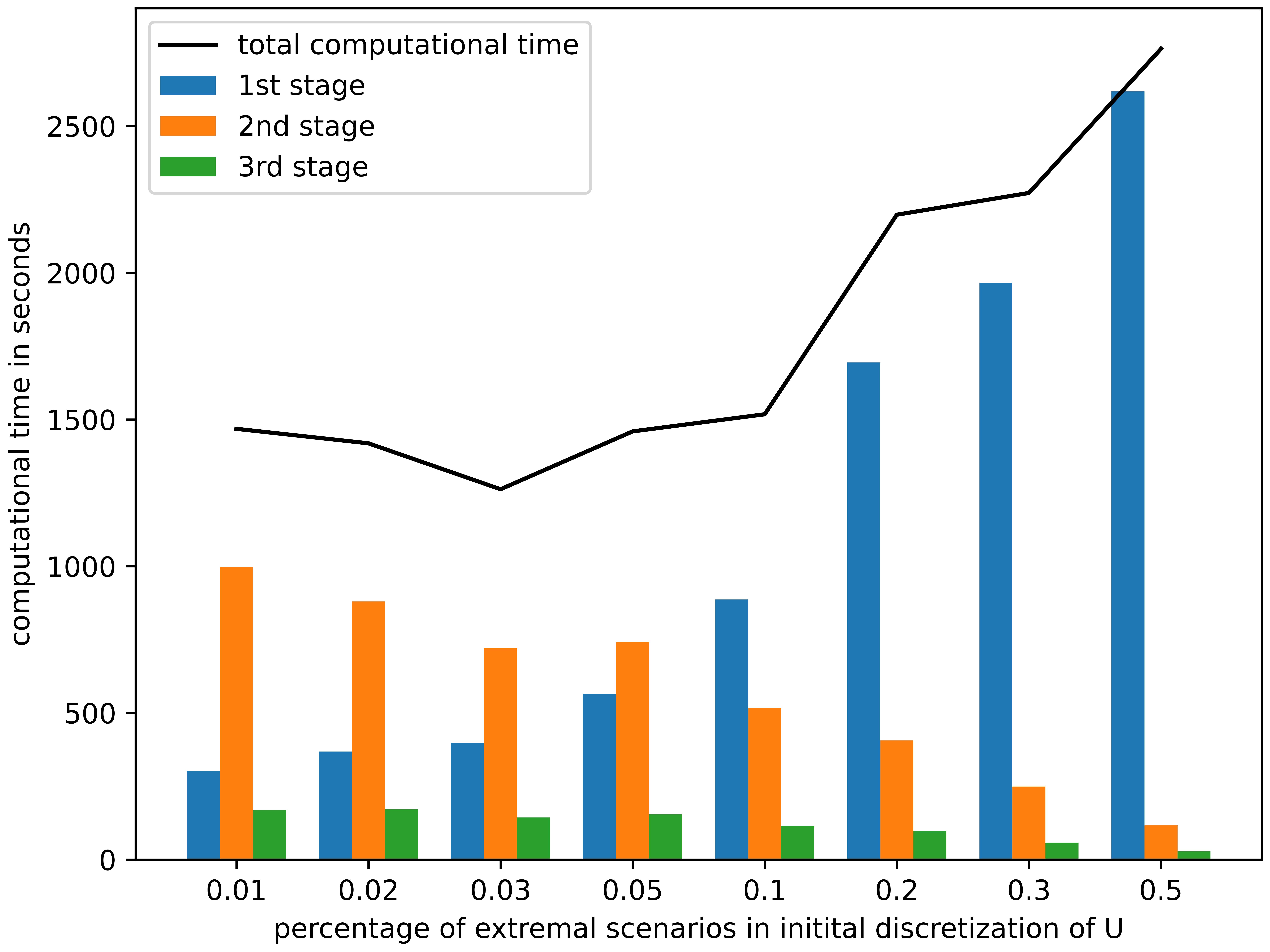}
		\caption{The computational time of the \nth{2} and \nth{3} stage of Algorithm \ref{alg:3_stage_algo} decreases for a increasing number of extremal scenario in the initial discretization of the uncertainty set, while the computational time of the \nth{1} stage increases.} 
		\label{fig:computation_time_3_stage_algo}
	\end{figure}
	In Figure \ref{fig:computation_time_3_stage_algo}, the mean of the computational time for each size of the initial discretization is shown. Each color represents one stage of Algorithm~\ref{alg:3_stage_algo} and the black line gives the total computational time of all three stages. 

    For the computational time spent in each stage, an interesting pattern can be seen. For small initial discretizations, the \nth{1} stage takes not much computational time. With increasing size of $\dot{\U}^1$, the computational time of the \nth{1} stage increases. In contrast to this, the computational time for the \nth{2} stage is high for small initial discretizations and decreases with increasing number of extremal scenarios in the initial discretization. While, for every size of the initial discretization, the least time is spent in the \nth{3} stage out of all stages, the computational time of the \nth{3} stage is also decreasing in the number of extremal scenarios in the initial discretization. 
    The reason for this pattern is that in the case of a small initial discretization, 
    the lower level problem \ref{opt_problem:subproblem_min_cost} must only be solved for a few scenarios in the initialization or Algorithm~\ref{alg:3_stage_algo}. This also implies that the discretized problem \ref{opt_prob:subproblem_min_max_regret} only considers a few scenarios. Therefore, it is often the case that there exist scenarios $\bmu \in \U \setminus \dot{\U}^1$, for which the solution $\bmpi^1$ of the discretized problem is not feasible. This can lead to multiple repetitions of the \nth{1} and \nth{2} stage 
    until the solution of the discretized problem $\bmpi^k$ is feasible for every scenario in the whole uncertainty set $\U$. This is the reason, why the \nth{2} stage takes the most time in the case of small initial discretizations. As the number of extremal scenarios in the initial discretization increases, more problems \ref{opt_problem:subproblem_min_cost} have to be solved in the initialization of Algorithm~\ref{alg:3_stage_algo}, which adds to the computational time of the \nth{1} stage. However, with the increasing size of the initial discretization, the solution of the discretized problem \ref{opt_prob:subproblem_min_max_regret} is feasible for more scenarios. Hence, the \nth{2} stage of Algorithm \ref{alg:3_stage_algo} is entered less often until the solution $\bmpi^k$ of the discretized problem is feasible for the whole uncertainty set $\U$, which leads to the decreasing computational time of this stage. 
	In the example considered here, there is a sweet spot of the size of the initial discretization at 3 percent of the extremal scenarios. However, this cannot be understood as a general rule. 
    The elements in the initial discretization were chosen randomly from all extremal scenarios. A more sophisticated rule for choosing the initial discretization could improve the computational time further. 
    
	Yet, it can be seen clearly that skipping the \nth{2} stage with the initial discretization containing all extremal scenarios, $\extremalpointsofU \subseteq \dot{\U}^1$, does not weigh off the increase in computational time of the \nth{1} stage: the savings in computational time for the \nth{2} stage are overcompensated by the increase in computational time for the \nth{1} stage. This explains why Algorithm \ref{alg:3_stage_algo} should be preferred over the algorithm proposed in \cite{Shimizu.1980} in the case of problems like the one considered here.

\subsection{Comparison of adjustable and min-max-regret robustness with an affine linear decision rule}\label{section:comparison_adjustable_minmaxregret}

In this subsection, min-max-regret robustness combined with an affine linear decision rule is compared to affinely adjustable robustness, where the worst-case costs are minimized. 
The comparison of the concepts is done by an ex-post analysis. For each concept, robust solutions are determined and the resulting costs and regret values are compared for different scenarios. Since both concepts use decision rules, the actual costs depend on the part of the uncertainty, to which the decision can be adjusted. Therefore, the costs can only be analyzed in hindsight after the uncertainty is revealed. 

For the two examples in this section, both, 
the adjustable and the min-max-regret problem are implemented in python with Pyomo \cite{Bynum.2021, Hart.2011}. As in before, the min-max-regret problem is solved with the proposed 3-stage algorithm using KNITRO 
 \cite{Byrd2006} for the subproblems in the \nth{1} stage, and Gurobi \cite{gurobi} in the \nth{2} and \nth{3} stage. The adjustable robust problem is solved with PyROS \cite{Isenberg.2021} with KNITRO as local solver and Gurobi as global solver.

First, consider Problem \eqref{opt_problem:results} with the parameters given in Table \ref{tab:parameter_settings_example_4_2}. In this problem, the decision, and hence the costs, depends on the uncertain parameters $u(1), u(2)$. In the application of drinking water supply systems, the demand varies around some nominal scenario. This is assumed to be given by $\bmu^{nom} = (900, 1700, 1500)$. When this scenario is likely to occur, the costs in this scenario, i.e., the nominal costs, can also be of interest for the user. 

    \begin{table}[]
    \centering
    \begin{tabular}{|l|llllll|}
		\hline
		$P$ &  \multicolumn{6}{l|}{2} \\
		$T$ & \multicolumn{6}{l|}{3}\\
		$\kappa$  & \multicolumn{6}{l|}{1} \\
		$N$ & \multicolumn{6}{l|}{10000}   \\
        $ \varepsilon$ &  \multicolumn{6}{l|}{0.00001 }\\
		$A$& \multicolumn{6}{l|}{1400} \\
		$h^{min}$& \multicolumn{6}{l|}{4.5} \\
		$h^{min,T}$& \multicolumn{6}{l|}{5}\\
	   $h(0)$ & \multicolumn{6}{l|}{6.3}  \\
	$h^{max}$ & \multicolumn{6}{l|}{6.5} \\
	 \hline \hline 
	 	 %\multicolumn{13}{|l|}{parameters for $t$} \\
	 	 $t$ & 1 & 2 & 3 &  & & \\ \hline 
	 	 $e(t)$ & 1 & 1.2 &0.8&& &  \\ 
	 	 $u^{min}(t)$ & 750.00 & 1226.80& 1168.83 & &  & \\ 
	 	 $u^{max}(t)$ &  1125.00 & 2278.35& 1948.05 &&&  \\ \hline  \hline 
	 	 %\multicolumn{3}{|l|}{parameters for $p$:} \\
	 	 $p$ & \multicolumn{2}{l}{1} & \multicolumn{4}{l|}{2} \\ \hline
	 	 $c_{p,2}$ & \multicolumn{2}{l}{0.000404} & \multicolumn{4}{l|}{0.0003} \\
	 	 $c_{p,1}$ & \multicolumn{2}{l}{-0.07334} & \multicolumn{4}{l|}{-0.06} \\
	 	 $c_{p,0}$ & \multicolumn{2}{l}{27.78} & \multicolumn{4}{l|}{20} \\
	 	 $Q(p)$ & \multicolumn{2}{l}{1800} & \multicolumn{4}{l|}{1300} \\
	 	 \hline
\end{tabular}
    \caption{Parameters of Problem \eqref{opt_problem:results} considered in the first example in Section \ref{section:comparison_adjustable_minmaxregret}.}
    \label{tab:parameter_settings_example_4_2}
\end{table}

 In Figure \ref{fig:cost_comparison_adjustable_minmaxregret} it can be seen, which model performs best (with respect to the costs) for the possibly occurring scenarios $u(1), u(2)$.
    \begin{figure}
        \centering
        \includegraphics[width=.8\textwidth]{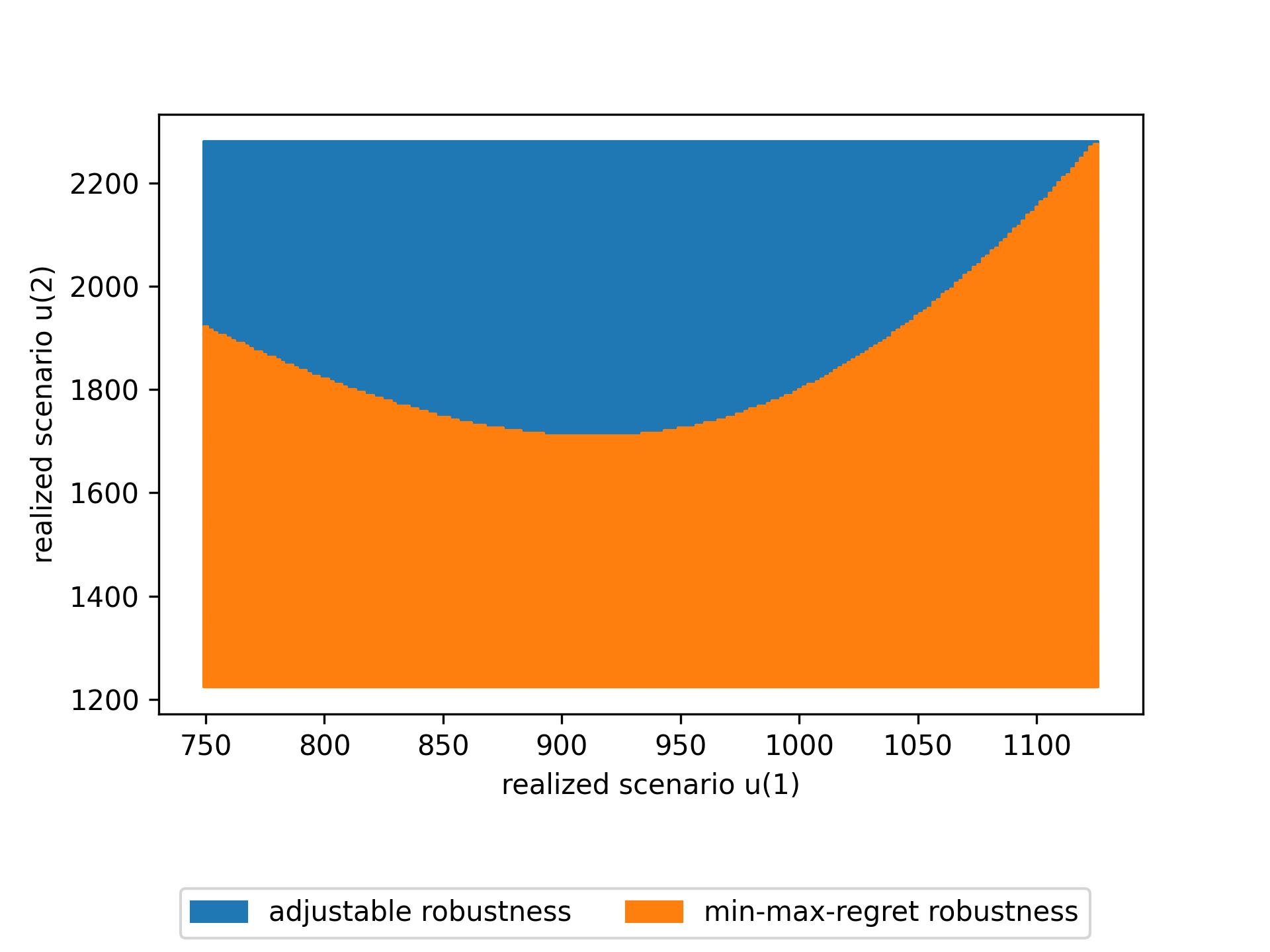}
        \caption{For both concepts, adjustable and recoverable robustness, there exist scenarios, where this concept is preferable in an ex-post comparison.}
        \label{fig:cost_comparison_adjustable_minmaxregret}
    \end{figure}
    It can be seen that both concepts have regions where their decision results in lower costs than the decision of the other robustness concept. Adjustable robustness performs better than min-max-regret robustness in many scenarios with a high value of $u(2)$. However, min-max-regret robustness also performs well when both uncertain parameters, $u(1)$ and $u(2)$, are high. 
    Figure \ref{fig:cost_comparison_adjustable_minmaxregret} shows the preferable regions for each scenario, but not the absolute values of the costs. Table \ref{tab:cost_comparison_adjsutable_minmaxregret} gives the worst-case costs, the nominal costs and the maximal regret of the adjustable and min-max-regret robust decisions. 
    \begin{table}[]
        \centering
        \begin{tabular}{|l|c|c|}
        \hline
        & adjustable robustness & min-max-regret robustness\\ \hline
             worst-case cost 	 &	 616.962	& 616.9629	\\
             & (best) & ($+0.0001$ \%)\\ \hline
nominal cost 	&433.1376	& 430.8811	\\ 
 & ($ + 5.2369$ \%)& (best) \\ \hline
maximal regret 	 	& 249.2559	 & 227.2854	 \\ 
& ($+ 9.6665 $ \%) & (best)\\ \hline
        \end{tabular}
        \caption{In the example, adjustable robustness and min-max-regret robustness have almost the same worst-case cost, while min-max-regret has considerably lower nominal cost and maximal regret.}
        \label{tab:cost_comparison_adjsutable_minmaxregret}
    \end{table}

It can be seen that both concepts have almost the same worst-case cost. However, in the nominal scenario, the min-max-regret robust decision results in lower costs. And, when the maximal regret over all scenarios is considered, min-max-regret robustness outperforms adjustable robustness considerably. 
The difference of the two concepts in the latter two criteria is even more dramatically in the next example. 

Now, the two robustness concepts are compared for a larger example. Consider Problem \eqref{opt_problem:results} with the parameters as given in Table \ref{tab:parameter_settings_example_4_2_2} and with the nominal scenario ${u^{nom}= (900, 1200, 1500, 1200, 850, 1000, 900)}$. 
    \begin{table}[]
    \centering
    \begin{tabular}{|l|lllllll|}
		\hline
		$P$ &  \multicolumn{7}{l|}{1} \\
		$T$ & \multicolumn{7}{l|}{7}\\
		$\kappa$  & \multicolumn{7}{l|}{2} \\
		$N$ & \multicolumn{7}{l|}{10000}   \\
        $ \varepsilon$ &  \multicolumn{7}{l|}{0.000001 }\\
		$A$& \multicolumn{7}{l|}{1400} \\
		$h^{min}$& \multicolumn{7}{l|}{4.5} \\
		$h^{min,T}$& \multicolumn{7}{l|}{5}\\
	   $h(0)$ & \multicolumn{7}{l|}{5.5}  \\
	$h^{max}$ & \multicolumn{7}{l|}{7} \\
	 \hline \hline 
	 	 %\multicolumn{13}{|l|}{parameters for $t$} \\
	 	 $t$ & 1 & 2 & 3 & 4 & 5& 6 & 7\\ \hline 
	 	 $e(t)$ & 1.0 & 1.0 &0.8& 1.0 & 1.0 & 1.0 & 1.0  \\ 
	 	 $u^{min}(t)$ & 750.00 &  865.98 & 1168.83 &  979.59 &  772.73 &
  909.09 & 734.69\\ 
	 	 $u^{max}(t)$ & 1125.00 & 1608.25 & 1948.05 & 1469.39 &   944.44 &
 1111.11 & 1102.04  \\ \hline  \hline 
	 	 %\multicolumn{3}{|l|}{parameters for $p$:} \\
	 	 $p$ & \multicolumn{2}{l}{1} & \multicolumn{5}{l|}{} \\ \hline
	 	 $c_{p,2}$ & \multicolumn{2}{l}{0.000404} & \multicolumn{5}{l|}{} \\
	 	 $c_{p,1}$ & \multicolumn{2}{l}{-0.07334} & \multicolumn{5}{l|}{} \\
	 	 $c_{p,0}$ & \multicolumn{2}{l}{27.78} & \multicolumn{5}{l|}{} \\
	 	 $Q(p)$ & \multicolumn{2}{l}{1800} & \multicolumn{5}{l|}{} \\
	 	 \hline
\end{tabular}
    \caption{Parameters of Problem \eqref{opt_problem:results} considered as second example in Section \ref{section:comparison_adjustable_minmaxregret}.}
    \label{tab:parameter_settings_example_4_2_2}
\end{table}

The resulting robust solutions are evaluated according to the three criteria of worst-case cost, nominal cost and maximal regret. The costs are shown in Table \ref{tab:cost_comparison_adjustable_minmaxregret_2}
\begin{table}
	\centering
	\begin{tabular}{|l|c|c|c|}%{|l|p{3cm}|p{3cm}|p{3cm}|}
		\hline 
		& adjustable robustness & min-max regret robustness \\ 
		%&  &  \\ 
  \hline
		worst-case cost & 3708.5053  &3717.1894 \\
						& (best) &   ($+ 0.23$ \%) \\ \hline
		nominal cost & 2752.0929& 2581.0971 \\
		 			& ($+ 6.62$ \%) & (best) \\ \hline
		maximal regret & 837.8284 & 496.0199 \\ 
				& 	($+ 68.91 $\%) &  (best)  \\
 \hline
	\end{tabular}
	\caption{In the second example considered in Section \ref{section:comparison_adjustable_minmaxregret}, the min-max-regret robust decision results in slightly higher worst-case cost, but outperforms adjustable robustness w.r.t.\ the nominal cost and the maximal regret.}
	\label{tab:cost_comparison_adjustable_minmaxregret_2}
\end{table}
While adjustable robustness yields the lower worst-case cost, min-max-regret robustness outperforms with respect to the nominal cost as well as the maximal regret. Furthermore, the min-max-regret robust decision results in worst-case cost that are close to the worst-case cost of adjustable robustness. In contrast to this, adjustable robustness is signifcantly worse in the nominal scenario and even more so with respect to the maximal regret.

	\section{Discussion} \label{section:discussion}

The examples considered in Section \ref{section:comparison_adjustable_minmaxregret}, where the results of adjustable and min-max-regret robustness are compared are relatively small. When larger examples are considered, PyROS \cite{Isenberg.2021} may encounter difficulties and potentially fail to find solutions. 
For instance, solving the adjustable robust formulation of the problem considered in Section \ref{section:performance_algo_depending_on_initial_discretization} with 156 decision variables, 
using PyROS with Gurobi \cite{gurobi} as the global solver and either KNITRO \cite{Byrd2006} or Gurobi as the local solver did not yield a robust solution. In contrast, the min-max-regret version of the same problem was successfully solved by the algorithm proposed in this paper. 
For practitioners facing large-scale optimization problems with inherent uncertainties, the traditional optimization approach implemented in PyROS may prove inadequate. In such cases, it seems advantageous to consider min-max-regret robustness instead of the worst-case approach. Then, Algorithm~\ref{alg:3_stage_algo} can be applied and generates a robust solution where none was found. 

Despite the algorithm's ability to handle large problems effectively, 
there is potential for further enhancing computational performance. In the example in Section~\ref{section:performance_algo_depending_on_initial_discretization}, the computational time was compared for different sized uncertainty sets with randomly chosen extremal scenarios. One idea for improvement could involve refining the selection criteria for the initial discretization. A promising strategy might be to choose scenarios that are significantly distinct from others in the set, based on a suitable measure of distance. Exploring such criteria could lead to even more efficient solutions and is an intriguing direction for future research.

 \section{Conclusion}\label{section:conclusion}
 In this work, the concepts of adjustable robust optimization and min-max-regret robust optimization were combined into a new class of robust optimization problem. The combination of the two concepts allows for robust solutions that can be adjusted to realized scenarios and are less conservative than worst-case-cost optimization.  
 A convergent three-stage algorithm was developed for the min-max-regret problem based on adaptive discretization of the uncertainty set and it was shown to be efficient.

The numerical examples provided within this work illustrate the superiority of the new concept over traditional adjustable robust optimization, particularly in its ability to adeptly handle large instances of min-max-regret problems.

	%\section*{Acknowledgments}
	
	\section*{Declaration of Interest Statement}
 The authors report there are no competing interests to declare.
	
	%\begin{thebibliography}{literatur}
	\bibliographystyle{plain}
	\bibliography{literatur}	
	%\end{thebibliography}
	
%	\section{Appendices}

\end{document}